\documentclass[12pt]{amsart}
\usepackage{fullpage}
\usepackage{amsmath}
\usepackage{amsthm}
\usepackage{amscd}
\usepackage{amssymb,MnSymbol,color,tikz}
\usepackage{graphicx}
\usepackage{amssymb}
\usepackage{epstopdf}
\usepackage{float}

\newtheorem{theorem}{Theorem}[section]

\newtheorem{lemma}[theorem]{Lemma}
\newtheorem{proposition}[theorem]{Proposition}

\theoremstyle{definition}
\newtheorem{definition}[theorem]{Definition}

\newtheorem{example}[theorem]{Example}

\newtheorem{algorithm}[theorem]{Algorithm}

\begin{document}


\title{Thurston's algorithm and rational maps from quadratic polynomial matings} \author{Mary Wilkerson}


\begin{abstract}

Topological mating is an combination that takes two same-degree polynomials and produces a new map with dynamics inherited from this initial pair. This process frequently yields a map that is Thurston-equivalent to a rational map $F$ on the Riemann sphere. Given a pair of polynomials of the form $z^2+c$ that are postcritically finite, there is a fast test on the constant parameters to determine whether this map $F$ exists---but this test is not constructive. We present an iterative method that utilizes finite subdivision rules and Thurston's algorithm to approximate this rational map, $F$. This manuscript expands upon results given by the Medusa algorithm in \cite{MEDUSA}. We provide a proof of the algorithm's efficacy, details on its implementation, the settings in which it is most successful, and examples generated with the algorithm.
\end{abstract}

\maketitle
\tableofcontents 
\addtocontents{toc}{\vskip-40pt} 

\let\thefootnote\relax\footnote{2010 \emph{Mathematics Subject ClassiÞcation}. Primary 37F20; Secondary 37F10.}
\let\thefootnote\relax


\vspace{-40pt}

\section{Introduction}

\emph{Mating} refers to a collection of operations that combine a pair of two same-degree polynomials in order to form a new map. Depending on the type of mating and polynomial pair, the resulting mating may be Thurston-equivalent to a rational map. When topologically mating two postcritically finite polynomials of the form $z\mapsto z^2+c$, the $c$ parameters determine if the resulting mating behaves like a rational map on the 2-sphere---but this test is not constructive of the rational map itself. \cite{LEI} \cite{REES} \cite{SHISHIKURA}

Thurston's topological characterization of rational maps, which is the driving force behind this parameter test, gives a more general criteria for when a topological map $g:\mathbb{S}^2\rightarrow\mathbb{S}^2$ is Thurston-equivalent to a rational map $F:\hat{\mathbb{C}}\rightarrow\hat{\mathbb{C}}$. The proof of Thurston's characterization given in \cite{TOPCHARACTERIZATION} suggests an algorithm for obtaining $F$: if we take Thurston pullbacks of a complex structure on $\mathbb{S}^2$ by $g$, this process incidentally outputs a sequence of rational maps converging to $F$. To take these pullbacks however, we must have some topological understanding of how the branched cover $g$ acts on various subsets of $\mathbb{S}^2$. Since this is sometimes difficult information to encode, few direct attempts (such as those in \cite{MEDUSA, SPIDER}) have been made at using Thurston's algorithm to find $F$. 

If we wish to find the rational map associated with a mating using Thurston's algorithm, it is clear that we need to understand the topological structure of the mated map first. In \cite{FSRCONSTRUCTION}, finite subdivision rules are constructed to develop this understanding. In the situations where a finite subdivision rule exists, we then can apply Thurston's algorithm to obtain an approximation to our desired map. The content of this article elaborates on this argument.

Put succinctly, we will develop a combinatorial map to model the behavior of the essential mating, and use this knowledge to assist in finding rational map approximations to the geometric mating. We start with prerequisite topics in Section \ref{prereqs}.  In this section, we will discuss quadratic polynomials, their matings, and the Thurston and Medusa algorithms. In Section \ref{yourfsr}, we develop a finite subdivision rule construction that is tailored to the goal of describing mapping behavior of certain matings.

In Section \ref{algorithms}, we introduce our main results: a method for obtaining rational maps from postcritically finite quadratic matings using finite subdivision rules and Thurston's algorithm. We prove that the output of our iterative algorithm  determines an approximation to the desired rational map, and highlight situations in which our algorithm extends the reach of a similar technique called the \emph{Medusa algorithm} \cite{MEDUSA}. We comment on examples and further avenues of exploration in Section \ref{connections}.

\section{Prerequisites in dynamics}\label{prereqs}

\subsection{Thurston equivalence}\label{thureq}

We have thus far discussed a notion of maps behaving in a dynamically similar fashion. We will formalize this with the definition below:

\begin{definition} Let $f,g: \mathbb{S}^2\rightarrow \mathbb{S}^2$ be two branched mappings with postcritical sets $P_f$ and $P_g$. We have that $f$ and $g$ are said to be \emph{Thurston equivalent} if and only if there exist homeomorphisms $h, h': (\mathbb{S}^2,P_f)\rightarrow(\mathbb{S}^2,P_g)$ such that

\begin{center}

$\begin{CD}
(\mathbb{S}^2,P_f)		@>h'>> 		(\mathbb{S}^2,P_g)\\
@VVfV			 						@VVgV\\
(\mathbb{S}^2,P_f)		@>h>> 		(\mathbb{S}^2,P_g)\\
\end{CD}
$
\end{center}

commutes, and such that $h$ is isotopic to $h'$ relative to $P_f$.\cite{TOPCHARACTERIZATION}

\end{definition}

When $f$ and $g$ are Thurston equivalent, this implies that the action of $f$ on a sphere containing its postcritical set is similar to the action of $g$ on a sphere containing its postcritical set.


\subsection{Parameter space}\label{parameter} 

Let $M$ denote the Mandelbrot set. The work in this paper will emphasize those parameters $c\in M$ for which the polynomials $f_c(z)=z^2+c$ are critically preperiodic. For such polynomials we have that $c\in M$, and that the associated Julia set $J_c$ is a connected and locally connected dendrite. When $J_c$ is a dendrite, the Julia set has no interior and we have that $J_c$ is equivalent to the filled Julia set, $K_c$.

When $K_c$ is connected, its complement on the Riemann sphere $\hat{\mathbb{C}}\backslash K_c$ is conformally isomorphic to $\hat{\mathbb{C}}\backslash\overline{\mathbb{D}}$ via some map $\phi:\hat{\mathbb{C}}\backslash \overline{\mathbb{D}}\rightarrow\hat{\mathbb{C}}\backslash K_c$. We have that $\phi$ is uniquely determined under the additional constraint that it conjugate $f_0:\hat{\mathbb{C}}\backslash\overline{\mathbb{D}}\rightarrow\hat{\mathbb{C}}\backslash\overline{\mathbb{D}}$ to $f_c:\hat{\mathbb{C}}\backslash K_c\rightarrow \hat{\mathbb{C}}\backslash K_c$ so that $\phi \circ f_0 = f_c\circ \phi$. 

\begin{figure}[hbt]
\center{\includegraphics[height=2in]{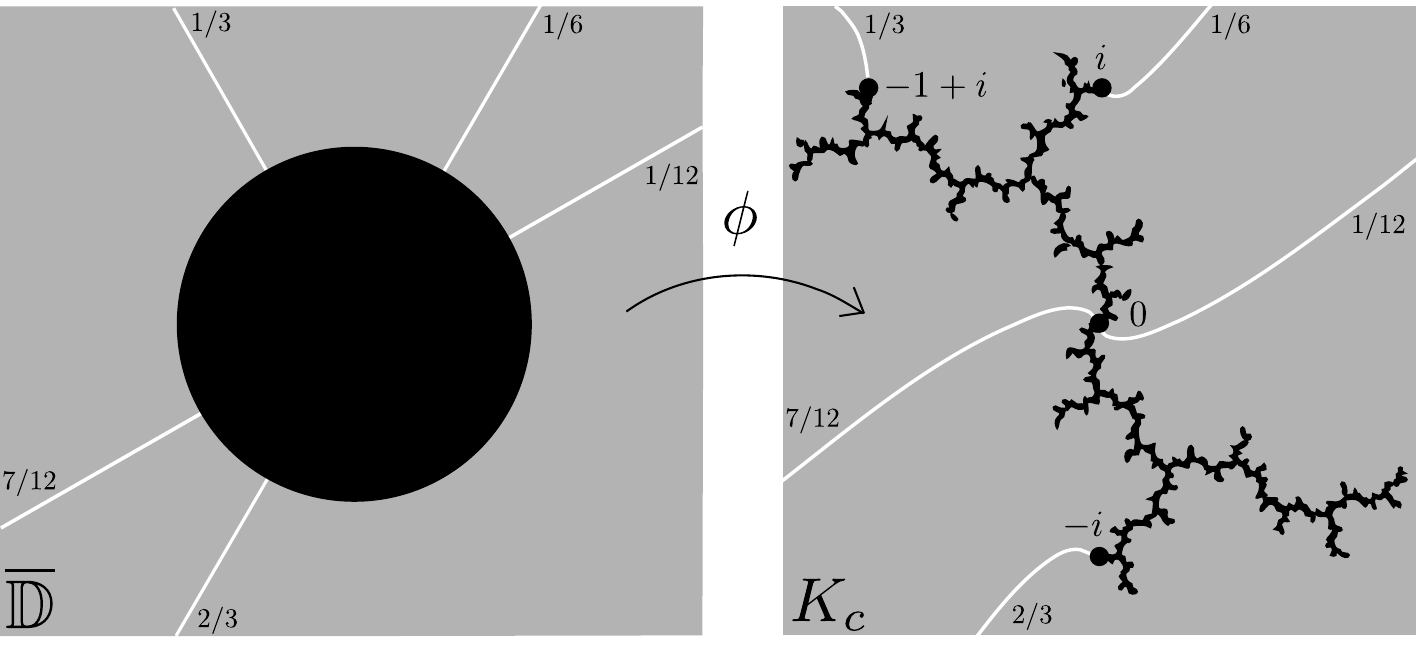}}
\caption{The conformal isomorphism $\phi$ which determines external rays for $z\mapsto z^2+i$. Shown on the right are external rays landing at points on the critical orbit of this polynomial.}
\label{varphi}
\end{figure}

We may use $\phi$ to define \emph{external rays of angle $t$}, $R_c(t)$ by fixing $t\in \mathbb{R}/\mathbb{Z}$ and taking images of rays around the unit disk under $\phi$,  $R_c(t)=\{\phi(re^{2\pi i t})|r\in(1,\infty)$, as demonstrated in Figure \ref{varphi}. Since the $K_c$ we discuss here are locally connected , we may take that $\phi$ extends continuously to $\partial\mathbb{D}$ and that external rays of angle $t$ have \emph{landing point} given by $\gamma(t)=\displaystyle\lim_{r\rightarrow 1^+}\phi(re^{2\pi i t})$. The map $\gamma$ is called the \emph{Carath\'{e}odory semiconjugacy}, which in the degree two case highlights the angle-doubling behavior of the map $f_c$ when applied to landing points of external rays: $\gamma(2t)=f_c(\gamma(t))$. This is emphasized on the right of Figure \ref{varphi} for the polynomial $z\mapsto z^2+i$: we may note the critical orbit portrait $$0\mapsto i\mapsto -1+i \mapsto -i$$ for this map, or we may double the angles of external rays and record the locations of landing points in order to observe the same behavior.

Given $\phi$, a typical notational convention is to parameterize critically preperiodic polynomials $z\mapsto z^2+c$ by an angle $\theta$ of an external ray landing at the critical value rather than by $c$. In the event that more than one ray lands at the critical value, there may be multiple parameters referring to the same polynomial. As a simpler example, the polynomial given in Figure \ref{varphi} could be named $f_{1/6}$ rather than $f_i$. We will adopt this $f_\theta$  convention in lieu of the use of $f_c$ for the remainder of the paper.

\subsection{Matings}\label{mating}  

Let $\widetilde{\mathbb{C}}$ be the compactification of $\mathbb{C}$ formed by union with the circle at infinity, $\widetilde{\mathbb{C}}=\mathbb{C}\cup\{\infty\cdot e^{2\pi i \theta}|\theta \in \mathbb{R}/\mathbb{Z}\}$. Then, take two monic same-degree polynomials with locally connected and connected filled Julia sets acting on two disjoint copies of $\widetilde{\mathbb{C}}$. If we form an equivalence relation on these copies of $\widetilde{\mathbb{C}}$ appropriately, the polynomial pair will determine a map that descends to the quotient space. This map is the \emph{mating} of the two polynomials. There are many kinds of polynomial matings, each dependent upon the equivalence relation we select. We will discuss three fundamental constructions: \emph{formal} matings, \emph{topological} matings, and \emph{essential} matings.

\begin{definition}Let $f_\alpha:\widetilde{\mathbb{C}}_\alpha\rightarrow\widetilde{\mathbb{C}}_\alpha$ and $f_\beta:\widetilde{\mathbb{C}}_\beta\rightarrow\widetilde{\mathbb{C}}_\beta$ be postcritically finite monic quadratic polynomials taken on two disjoint copies of $\widetilde{\mathbb{C}}$, and let $\sim_f$ be the equivalence relation which identifies $\infty \cdot e^{2\pi i t}$ on $\widetilde{\mathbb{C}}_\alpha$ with $\infty \cdot e^{-2\pi i t}$ on $\widetilde{\mathbb{C}}_\beta$ for all $t\in \mathbb{Z}/\mathbb{Z}$. Then, the quotient space $\widetilde{\mathbb{C}}_\alpha\bigsqcup \widetilde{\mathbb{C}}_\beta/\sim_f$ may be identified with $\mathbb{S}^2$, as this quotient glues two $\tilde{\mathbb{C}}$ disks together along their boundaries with opposing angle identifications to form a topological 2-sphere. (See Figure \ref{green}.) This quotient space serves as the domain of the \emph{formal mating} $f_\alpha\upmodels_ff_\beta$, that is the map that applies $f_\alpha$ and $f_\beta$ on their respective hemispheres of $\mathbb{S}^2$. \end{definition}

It should be noted that the Carath\'{e}odory semiconjugacy guarantees that $f_\alpha\upmodels_ff_\beta$ is well-defined on the equator, and provides a continuous branched covering of $\mathbb{S}^2$ to itself. 

\begin{figure}[h]
\center{\includegraphics[height=2in]{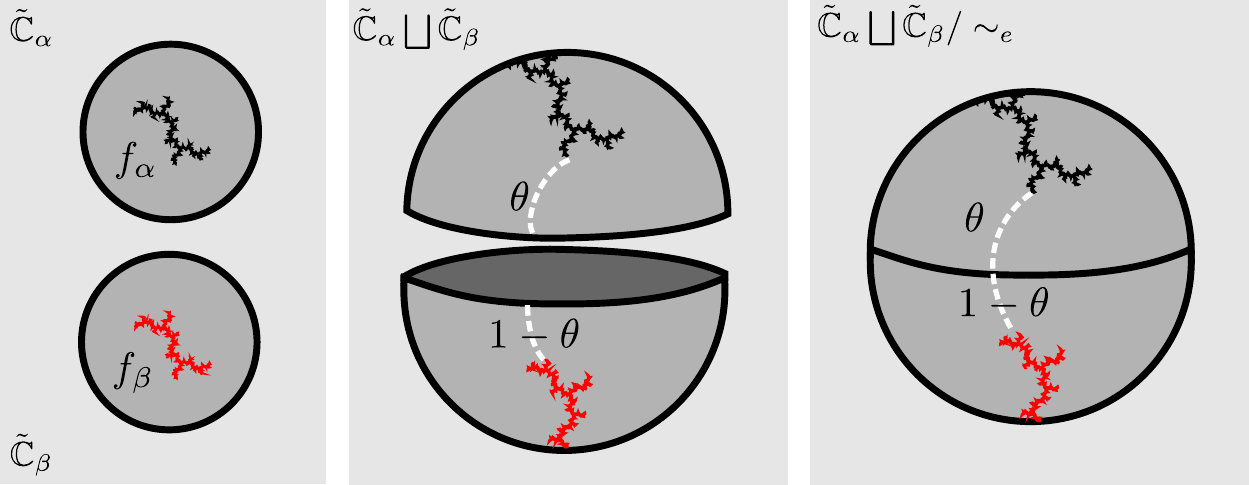}}
\caption{Steps in the formation of the formal mating.}
\label{green}
\end{figure}

\begin{definition}The domain of the \emph{topological mating} $f_\alpha\upmodels f_\beta$ is given by the quotient space $K_\alpha\bigsqcup K_\beta/\sim$, where $\sim$ identifies the landing point of $R_\alpha(t)$ on $J_\alpha$ with the landing point of $R_\beta(-t)$ on $J_\beta$. This quotient glues the filled Julia sets of $f_\alpha$ and $f_\beta$ together by their boundaries using opposing external angle identifications. Much like the formal mating, we obtain the map $f_\alpha\upmodels f_\beta$ by applying $f_\alpha$ and $f_\beta$ on their respective filled Julia sets. \end{definition}

The Carath\'{e}odory  semiconjugacy again guarantees that the topological mating is well-defined and continuous, but it is possible that it no longer acts on a quotient space that is a topological 2-sphere even though there is an induced map. When the domain is a 2-sphere is a solved problem for the postcritically finite quadratic case, noted in the following theorem:

\begin{theorem}[Lei, Rees, Shishikura]\label{LRS} The topological mating of the postcritically finite maps $z\mapsto z^2+c$ and $z\mapsto z^2+c'$ is Thurston-equivalent to a rational map on $\hat{\mathbb{C}}$ if and only if $c$ and $c'$ do not lie in complex conjugate limbs of the Mandelbrot set \cite{LEI}, \cite{REES}, \cite{SHISHIKURA}. 

\end{theorem}

In the event that $c$ and $c'$ are not in complex conjugate limbs of $M$, then the rational map referenced in the theorem above is called a \emph{geometric mating}. We will use $F$ to denote the geometric mating of a polynomial pair whenever it is unambiguous to do so.

Theorem \ref{LRS} is a powerful result as it allows us to make statements regarding the dynamics of the topological mating based on parameters alone. However, while this theorem tells us when a topological mating behaves like a rational map $F$ on the Riemann sphere, it is not constructive of this map. We will detail an approximation for $F$ later, but the algorithm will depend on having an understanding of the action of the topological mating on the sphere. Ideally we would use the formal mating instead, since it is much simpler in construction than the topological mating---but these two maps are not always Thurston equivalent. Instead, we shall make use of an intermediate mating operation called the \emph{essential mating}, $f_\alpha\upmodels_ef_\beta$. Starting with the quotient 2-sphere $\mathbb{S}^2$ developed in the formal mating, the essential mating is constructed as detailed below and in \cite{LEI}.

\begin{definition}\label{essential} Suppose $f_\alpha$ and $f_\beta$ are two polynomials with the properties described in Theorem \ref{LRS} whose topological mating is Thurston-equivalent to a rational map. Allow $\mathbb{S}^2$ to denote the domain of the formal mating $h=f_\alpha\upmodels_f f_\beta$.  We define the \emph{essential mating} using the following steps. 
\begin{enumerate}
\item Let $\{l_1,...,l_n\}$ be the set of maximal connected graphs of external rays on $\mathbb{S}^2$ containing at least two points of the postcritical set $P_h$, and let $\{q_1,...,q_m\}$ be the set of connected graphs of external rays in $\displaystyle\bigcup_{k=1}^\infty\bigcup_{i=1}^nh^{-k}(l_i)$ containing at least one point on the critical orbit of $h$. Take each of the $\{q_1,...,q_m\}$ to be an equivalence class of $\sim_e$, and note that $\mathbb{S}'^2 =\mathbb{S}^2/\sim_e$ is homeomorphic to a sphere since none of the equivalence classes of $\sim_e$ contain closed curves.  

\item Note that $h$ maps equivalence classes to equivalence classes, so letting $\pi:\mathbb{S}^2\rightarrow\mathbb{S}'^2$ denote the natural projection yields that $\pi\circ h\circ\pi^{-1}$ is well-defined and preserves the mapping order of equivalence classes.

\item Set $V_j$ to be an open neighborhood of $q_j$ such that $V_j\cap(P_h\cup\Omega_h)=q_j\cap(P_h\cup\Omega_h)$ for each $j$, and such that distinct $V_j$ are nonintersecting. For each $j$, denote by $\{U_{ij}\}$ the set of connected components of $h^{-1}(V_j)$ for which $U_{ij}\cap\displaystyle\bigcup_{p=1}^mq_p=\emptyset$. 

\item Define $f_\alpha\upmodels_ef_\beta:\mathbb{S}'^2\rightarrow\mathbb{S}'^2$ as follows. On the complement of $\displaystyle\bigcup_{i,j}\pi(U_{ij})$, we set $f_\alpha\upmodels_ef_\beta:=\pi\circ h\circ \pi^{-1}$. For each $i,j$ we define $f_\alpha\upmodels_ef_\beta:\pi(U_{ij})\rightarrow\pi(V_j)$ as some homeomorphism that extends continuously to the boundary of $\pi(U_{ij})$.  
\end{enumerate}

The map $f_\alpha\upmodels_ef_\beta$ is the \emph{essential mating} of $f_\alpha$ and $f_\beta$.

\end{definition} 

To unpack the definition, the action of the essential mating on the 2-sphere is similar to that of the formal mating, save for two changes.  First, postcritical points that fall into shared equivalence classes under $\sim_t$ are collapsed. Second, after collapsing along these `essential' equivalence classes, we modify the map slightly so that it does not map arcs to points and thus remains a branched covering. In making these changes, the essential mating retains much of the simplicity of the structure of the formal mating, but also serves as a map that is guaranteed to be Thurston equivalent to the topological mating--regardless of how arbitrary the selected homeomorphisms appear on the last step of the definition. \cite{LEI}

A notable implementation of this construction occurs when the postcritical points of a polynomial pairing fall into distinct equivalence classes of $\sim_t$. In this case, none of the postcritical points for the formal mating can be connected by a graph of external rays on $\mathbb{S}^2$, and so $\sim_e$ is trivial. Then, $\pi$ acts much like the identity and there are no $U_{ij}$, so we have that $f_\alpha\upmodels_ef_\beta=\pi\circ h\circ\pi$ on all of $\mathbb{S}'^2$. More simply, the essential and formal matings are the same map whenever polynomial postcritical points are not identified under $\sim_t$. If the essential and formal matings are the same for a pair of polynomials, we say that those polynomials are \emph{strongly mateable}.

To simplify notation from this point on, we will use $g$ to refer to an essential mating if it is clear to do so. As $\mathbb{S}'^2$ is homeomorphic to $\mathbb{S}^2$, we will further simplify notation by treating $g$ as a self-map on $\mathbb{S}^2$.

 \subsection{The Thurston and Medusa algorithms}\label{thurstonmedusa}

 Let $\mathcal{C}$ denote the space of orientation preserving complex structures on $(\mathbb{S}^2,P_f)$. We then define the \emph{Teichmuller space}, $\mathcal{T}_f$, to be the quotient of $\mathcal{C}$ by the group of orientation preserving diffeomorphisms of $(\mathbb{S}^2,P_f)$ that are isotopic to the identity. More specifically, we take two complex structures $\sigma_1, \sigma_2\in\mathcal{C}$ to be representatives of the same element $\tau\in\mathcal{T}_f$ if $\sigma_1=\sigma_2\circ g$ where $g$ is some orientation preserving homeomorphism on $\mathbb{S}^2$ that is isotopic to the identity relative to $P_f$.

Let $\sigma\in\mathcal{C}$.  The pullback of $\sigma$ under $f$ gives a complex structure on $\mathbb{S}^2$, and so the mapping $\sigma\mapsto \sigma(f)$ induces a holomorphic mapping $\Sigma_f:\mathcal{T}_f\rightarrow\mathcal{T}_f$ on Teichmuller space. We call $\Sigma_f$ the \emph{Thurston pullback map}.

We then have the following:

\begin{proposition}[Thurston, Douady, Hubbard] The mapping $f$ is Thurston-equivalent to a rational function if and only if $\Sigma_f$ has a fixed point. \cite{TOPCHARACTERIZATION}
\end{proposition}

This proposition is a necessary step en route to proving part of Thurston's topological characterization of rational maps. Thurston's theorem specifies that a critically finite branched map with hyperbolic orbifold is Thurston-equivalent to a rational function if and only if for any $f$-stable multi-curve $\Gamma$, the largest eigenvalue of the associated Thurston linear transformation is $<1$. Given this property, $\Sigma_f$ is holomorphic and thus distance non-increasing, while $\Sigma_f\ ^{\circ 2}$ is strictly contracting. If $\tau\in\mathcal{T}_f$, the sequence $\tau_n=\Sigma_f\ ^{\circ n}(\tau)$ then converges to the fixed point of $\Sigma_f$ in Teichmuller space, and so $f$ is Thurston-equivalent to a rational map. The interested reader may refer to \cite{TOPCHARACTERIZATION} for a full statement of the topological characterization of rational maps, and details on the proof referenced here.


The  process of constructing the sequence $\tau_n$ to find the rational map $F$ which is Thurston-equivalent to $f$ is referred to as \emph{Thurston's algorithm}. Since the $\tau_n$ are equivalence classes of complex structures, it is typical to work with representatives $\sigma_n\in\mathcal{C}$ that have been normalized in some way. We then obtain a sequence of rational maps $F_n=\sigma_n\circ f\circ\sigma_{n+1}^{-1}$ as in Figure \ref{commutative1} that converge to the desired $F$.

In theory, Thurston's algorithm is very straightforward. In practice, normalizing our complex structures and labeling critically finite branched maps in a manner stringent enough to obtain the pullback sequence $\{\sigma_n\}$ is difficult. A crux of the problem is that when $f$ is a degree $m$ branched map, all points except the critical values have $m$ preimages under $f$---and constructing a pullback relies on understanding the action of $f$ on $\mathbb{S}^2$ well enough to distinguish between these preimage points. Thus, unless we understand the mapping behavior of $f$ we cannot build a meaningful pullback map. This problem is not insurmountable however, as Thurston's algorithm has been successfully tailored to specific kinds of branched maps: the Spider algorithm is one such adaptation for polynomials \cite{SPIDER}.

Since matings are of particular interest to us, we'll give a brief overview of another adaptation of Thurston's algorithm for matings: the Medusa algorithm \cite{MEDUSA}. This algorithm iteratively approximates the rational map that is Thurston-equivalent to the topological mating of two quadratic polynomials. The commutative diagram in Figure \ref{commutative1} emphasizes the relationship of the maps involved. The Medusa algorithm starts by encoding the mapping structure of the associated formal mating $f:\mathbb{S}^2\rightarrow\mathbb{S}^2$ using the \emph{Medusa}: a 1-skeleton structure embeddable in $\mathbb{S}^2$ that contains the equator and all external rays which meet $P_f$. The Caratheodory angle-doubling semiconjugacy specifies the expected action of $f$ on key regions of $\mathbb{S}^2$: notably, the equator and `legs' that form the Medusa. 

\begin{figure}[htb]
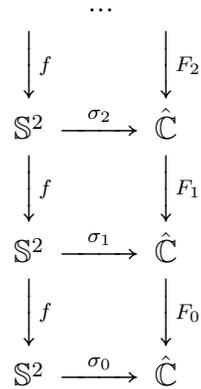

\center{$. . .$

\vspace{.1in}

$\begin{CD}
@VVfV			 											@VVF_2V\\
\mathbb{S}^2		@>\sigma_2>> 		\hat{\mathbb{C}}\\
@VVfV			 											@VVF_1V\\
\mathbb{S}^2		@>\sigma_1>> 		\hat{\mathbb{C}}\\
@VVfV			 											@VVF_0V\\
\mathbb{S}^2		@>\sigma_0>> 			\hat{\mathbb{C}}
\end{CD}
$
}
\caption{The Medusa and pseudo-equator algorithms are based upon Thurston's algorithm, highlighted in the commutative diagram above.}
\label{commutative1}
\end{figure}

Next, we specify an embedding $\sigma_0$ of the Medusa into $\hat{\mathbb{C}}$ that sends the equator of $\mathbb{S}^2$ to the unit circle, and prescribes desirable images for the legs. The embedding of the two critical values of $f$ can then be used as parameters determining a rational map $F_0:\hat{\mathbb{C}}\rightarrow\hat{\mathbb{C}}$ in a particular normalized form. Pulling back the embedded Medusa by $F_0$ provides a new Medusa structure embedded in $\hat{\mathbb{C}}$, that we take to induce a new embedding of the Medusa, $\sigma_1$. We iterate this process to develop the sequences $\sigma_n$ and $F_n$, where $F_n$ serves as an approximation to the rational map that is the mating of the two polynomials.

Here, we emphasize the motivation behind this work. The Medusa algorithm guarantees convergence of $F_n$ to the geometric mating of polynomials $f_\alpha$ and $f_\beta$ in the event that $f_\alpha$ and $f_\beta$ are strongly mateable.  In \cite{MEDUSA} however, it is noted that the algorithm fails in cases where postcritical points of the formal mating are identified in the topological mating. Indeed, Thurston's algorithm is intended for the setting where $f$ and the desired rational map are Thurston-equivalent---which is not the case when the two polynomials are \emph{not} strongly mateable. Investigating convergence of points at the boundary of Teichmuller space is one possible avenue of approaching this problem, but we suggest something more direct---start with a map $g$ that collapses the necessary points, and is Thurston-equivalent to the desired rational map. This is the essential mating constructed by Tan Lei in \cite{LEI}, and referenced earlier in Section \ref{mating}. Using this map instead of the formal mating would force Thurston's algorithm to generate a convergent sequence. 

Changing the map that we use in conjunction with Thurston's algorithm, however, means that we may no longer use the Medusa structure to model our mating. To overcome this obstacle, we will substitute a different method that provides a combinatorial map representative of the mating. We will thus conclude our discussion of dynamics prerequisites to begin a description of this model.
 
 
 \section{Finite subdivision rules and tilings}\label{yourfsr}
 
 As a reminder, our ultimate goal in this paper will be to iteratively approximate rational maps that are Thurston-equivalent to topological matings. To do so later will require an understanding of the mapping behavior of the essential mating. This section will detail the use of \emph{finite subdivision rules} to construct a model for this behavior.
 

 \subsection{Finite subdivision rules}\label{fsr}  The mapping behavior of complex functions is often demonstrated via a distortion of gridlines: by embedding a grid or tiling in $\mathbb{C}$ and observing how the image or preimage of the complex function distorts the tiling, we visualize the action of our function on the complex plane. We will utilize a similar but more specialized tool to study our mapping behavior, described below.
 
 \begin{definition} A \emph{finite subdivision rule} $\mathcal{R}$ is composed of the following three elements: 

\begin{enumerate}

\item A finite 2-dimensional CW complex $S_\mathcal{R}$, called the \emph{subdivision complex}, with fixed cell structure so that $S_\mathcal{R}$ is the union of its closed 2-cells. Each closed 2-cell $\tilde{s}$ of $S_\mathcal{R}$ must have a CW structure $s$ on a closed 2-disk so that $s$ has $\geq 3$ vertices, the vertices and edges of $s$ are contained in $\partial s$, and the characteristic map $\psi_s:s\rightarrow S_\mathcal{R}$ which maps onto $\tilde{s}$ restricts to a homeomorphism on open cells. More colloquially, we will refer to the subdivision complex as a \emph{tiling}.

\item A finite 2-dimensional CW complex $\mathcal{R}(S_\mathcal{R})$ which is a subdivision of $S_\mathcal{R}$. We will refer to this as a \emph{subdivided tiling}.

\item A continuous cellular map $g_\mathcal{R}: \mathcal{R}(S_\mathcal{R})\rightarrow S_\mathcal{R}$, which maps open cells of $\mathcal{R}(S_\mathcal{R})$ homomorphically to $S_\mathcal{R}$. Such a map $g_\mathcal{R}$ is called a \emph{subdivision map}. 

\end{enumerate}

Such finite subdivision rules may be applied recursively to yield iterated subdivisions of $S_\mathcal{R}$. \textup{\cite{FSRS} }
\end{definition}

In essence, finite subdivision rules heavily parallel the grid distortion technique described above. The primary difference is that when we pull back the initial tiling for a finite subdivision rule, we do not obtain an arbitrarily distorted tiling. Instead, the result is a new tiling which is a subdivision of the original. 

For us, a finite subdivision rule will be a finite combinatorial rule for subdividing tilings on a 2-sphere. We will assume that our tilings can be formed by `filling in' the faces of connected finite planar graphs on this 2-sphere with open tiles that are topological polygons. Each edge of the tiling must be a boundary edge to some tile, and tiles are not allowed to be monogons or digons, but they may be non-convex. We will allow for extreme cases where single edges serve as two sides of the boundary of a single tile. As a rudimentary example of such a finite subdivision rule, consider the following example:

\begin{example} In Figure \ref{fig:fsr}, $\hat{\mathbb{C}}$ is oriented so that the marked points $0, 1,$ and $\infty$ all lie on the equator. The positive real axis with these marked points determines a graph which yields a tiling of $\hat{\mathbb{C}}$ by a single topological quadrilateral. If we take a preimage of this structure under the map $z\mapsto z^2$, we obtain a tiling that has two quadrilaterals---each of which maps homeomorphically onto the quadrilateral in the original tiling. Here, the structure on the left is our subdivision complex $S_\mathcal{R}$, the structure on the right is the subdivided tiling $\mathcal{R}(S_\mathcal{R})$, and the map $z\mapsto z^2$ is the subdivision map.

\begin{figure}[htb]
\center{\includegraphics[width=4in]{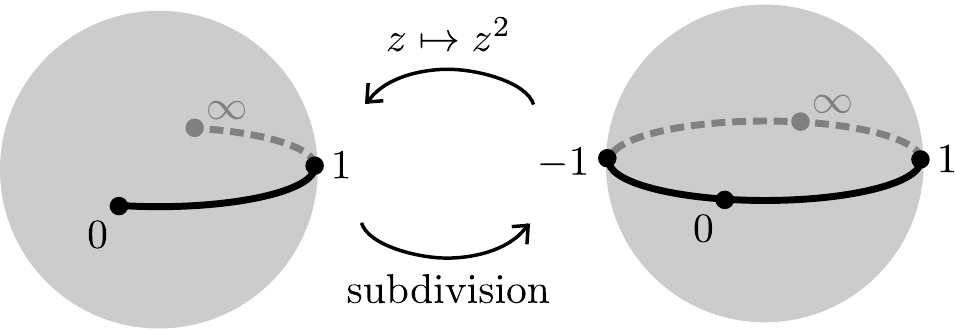}}
\caption{A rudimentary finite subdivision rule on $\hat{\mathbb{C}}$.}
\label{fig:fsr}
\end{figure}

\end{example}

While a finite subdivision rule may be defined using analytic maps and embedded tilings as in the previous example, this is not necessary. We can use the mapping behavior of $n$-cells in a tiling to inductively determine the mapping behavior of ($n+1$)-cells, thus obtaining a subdivision map based on combinatorial data. The reader may reference \cite{FSRS} for a more detailed treatment of this topic.

 \subsection{Hubbard trees}\label{Hubbard}

To build a finite subdivision rule that models the behavior of an essential mating, it will be helpful to have a finite invariant structure in mind to determine a tiling 1-skeleton. We start by considering invariant structures associated with the polynomial pair composing the mating. Julia sets are invariant under iteration of their associated polynomials, but the structure of a typical Julia set is not finite and hence cannot be used as a starting point for a finite subdivision rule. Thus, we would like to work with a discrete approximation to the Julia set: the \emph{Hubbard tree}.

The construction of a Hubbard tree for a polynomial $f_\theta$ can be simplified considerably in the case where  $f_\theta$ is critically preperiodic and has a dendritic Julia set--which is the primary setting for this paper. For the reader's convenience, we thus present a definition restricted to this situation below:

\begin{definition} Let $f_\theta: \mathbb{C} \rightarrow \mathbb{C}$ be given by $f_\theta(z) = z^2 + c$ for some Misiurewicz point $c$, and let $f_\theta$ have Julia set $J_\theta$ and postcritical set $P_{f_\theta}$.

We say that a subset $X$ of $J_\theta$ is \emph{allowably connected} if $x,y\in X$ implies that there is a topological arc in $X$ that connects $x$ and $y$. The \emph{allowable hull} of a subset $A$ in $J_\theta$ is then the intersection of all allowably connected subsets of $J_\theta$ that contain $A$. Finally, the \emph{Hubbard tree} of $f_\theta$ is the allowable hull of $P_{f_\theta}$ in $J_\theta$. \cite{ORSAY}
\end{definition}

\begin{figure}[htb]
\center{\includegraphics{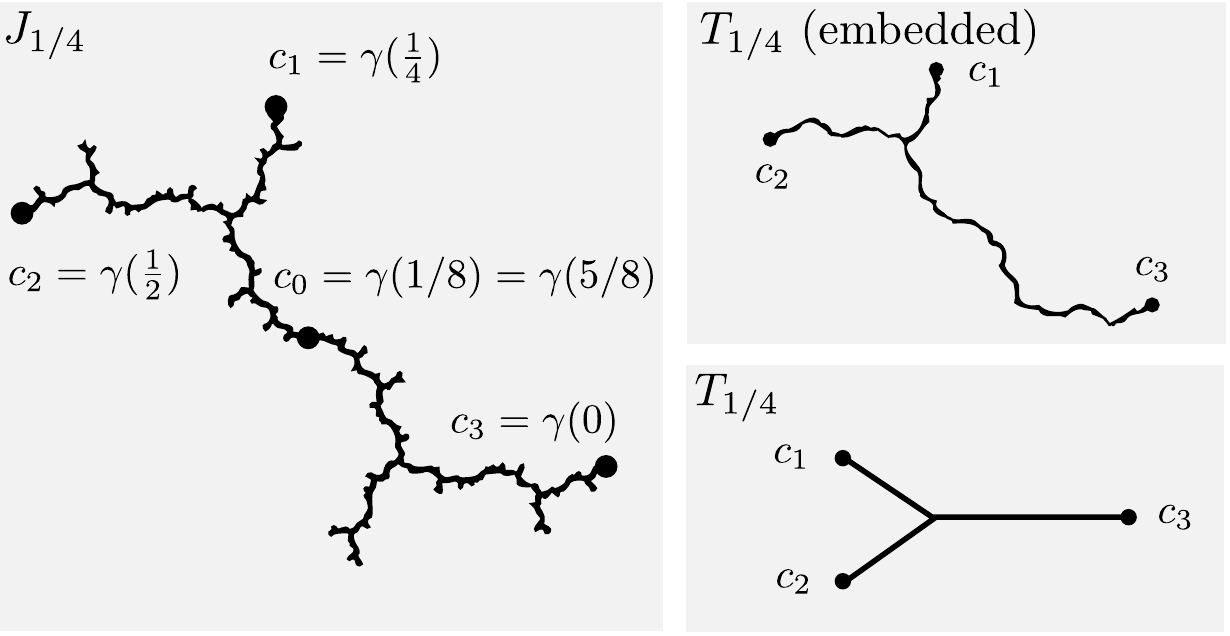}}
\caption{The Julia set and Hubbard trees for $f_{1/4}$.}
\label{hubbardtree}
\end{figure}

The Hubbard tree as defined above is embedded in $\mathbb{C}$ and topologically equivalent to the notion of an \emph{admissible Hubbard tree} with preperiodic critical point as discussed by Bruin and Schleicher in \cite{HUBBARDTREES}. These notes, however, emphasize the combinatorial structure of the Hubbard tree as a graph with vertices marked by elements of $P_{f_\theta}$, rather than as an embedded object in the complex plane. This distinction is emphasized on the right side of Figure \ref{hubbardtree}. Bruin and Schleicher present several explicit algorithms that can be used to construct a topological copy of $T_\theta$ from the parameter $\theta$, building heavily on the notion that quadratic maps are degree 2 at their critical points and behave locally homeomorphically elsewhere. We can further expand upon these observations regarding the behavior of quadratic polynomials to note the action of $f_\theta$ on $T_\theta$: $f_\theta$ acts locally homeomorphically on $T_\theta$ everywhere except at the critical point, which maps with degree two, and iterated preimages  of $T_\theta$ under $f_\theta$ give discrete approximations to $J_\theta$. We thus have that the $n$th preimage of a tree $T_\theta$ under its associated polynomial $f_\theta$ contains $2^n$ miniature copies of the tree that each map homeomorphically onto the tree via $f_\theta^{\circ n}$, as in Figure \ref{hubbardpreim}. Given a Hubbard tree and critical orbit portrait as in this figure, we may make note of the local homeomorphic behavior off of the critical point to `fill in' missing limbs of subsequent preimages of $T_\theta$. In this manner, we may then view $f_\theta$ as inducing a combinatorial map on trees. 

\begin{figure}[htb]
\center{\includegraphics{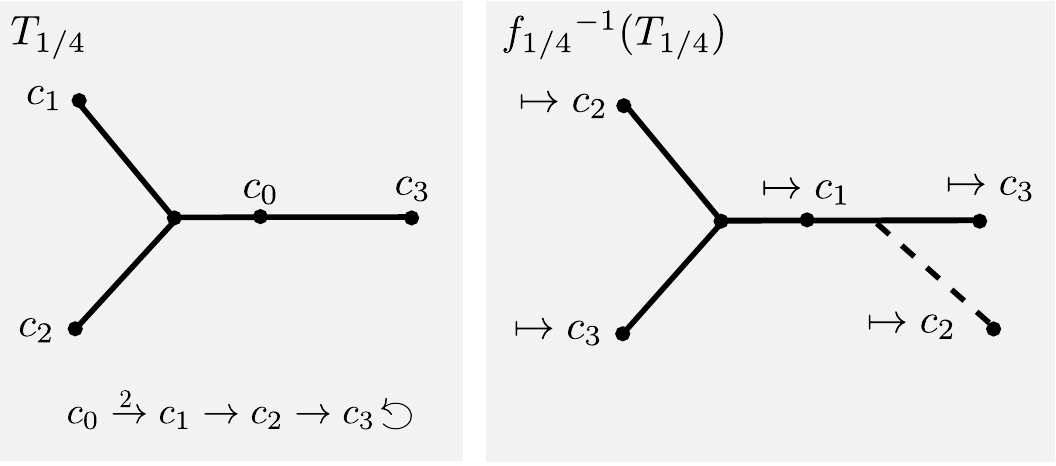}}
\caption{\label{hubbardpreim}{Preimages of a Hubbard tree under its associated polynomial.}}
\end{figure}

\subsection{Construction of a finite subdivision rule}\label{yourfsrrules}
As previously referenced, one of the impediments in using Thurston's algorithm to find geometric matings is the difficulty in recording the topological behavior of a map in a useful way. The goal of this section is to remove this obstacle and provide a combinatorial rule that records enough of the action of the essential mating on $\mathbb{S}^2$ for us to successfully apply Thurston's algorithm.

There are several possible ways to construct finite subdivision rules to model the essential mating of two critically preperiodic quadratic polynomials. Examples for strongly mateable polynomial pairs are given in \cite{DISSERTATION}. In the event that the formal and essential matings are different maps, $\sim_e$ will be a nontrivial equivalence relation and we may have the opportunity to use the construction below which is detailed in both \cite{DISSERTATION} and \cite{FSRCONSTRUCTION}.

\begin{definition}\label{fsrdefn} Suppose $f_\alpha$ and $f_\beta$ are critically preperiodic monic quadratic polynomials such that $x\sim_e y$ for some points $x\in T_\alpha,y\in T_\beta$.
\begin{enumerate}

\item Give $T_\alpha\bigsqcup T_\beta / \sim_e$ a graph structure on the quotient space of the essential mating by marking all postcritical points and branched points as vertices. If need be, mark additional periodic or preperiodic points on $T_\alpha$ or $T_\beta$ and the points on their forward orbits to avoid tiles forming digons. The associated 2-dimensional CW complex for this structure will yield the subdivision complex, $S_\mathcal{R}$.

\item Let $g$ denote the essential mating of $f_\alpha$ and $f_\beta$ and set $\mathcal{R}(S_\mathcal{R})$ to be the preimage of $S_\mathcal{R}$ under $g$, taking preimages of marked points of $S_\mathcal{R}$ to be marked points of $\mathcal{R}(S_\mathcal{R})$ .

\item If $\mathcal{R}(S_\mathcal{R})$ is a subdivision of $S_\mathcal{R}$ and if the essential mating $g:\mathcal{R}(S_\mathcal{R}) \rightarrow S_\mathcal{R}$ is a subdivision map, then $\mathcal{R}$ is a  \emph{finite subdivision rule construction of essential type}.
\end{enumerate}
\end{definition}

In a simplified sense, this finite subdivision rule construction examines how the Hubbard trees of two polynomials are glued together when forming the domain space of the essential mating. The glued pair of trees is a 1-skeleton for a tiling on the 2-sphere; the pullback of this tiling by the essential mating sometimes generates a subdivided tiling. This finite subdivision rule construction then may allow us to reduce the essential mating to a combinatorial map.

\begin{example}\label{f14matingfsrex}
Consider the essential mating of $f_{1/4}$ with itself. The Hubbard tree is presented on the left of Figure \ref{f14selfi}. The postcritical points of $f_{1/4}$ are the landing points of the $1/4$, $1/2$, and $0$ external rays. In the essential mating, the equivalence relation $\sim_e$ dictates that these landing points on opposing trees are identified using a $\theta$ and $1-\theta$ angle pairing. This means that in the self-mating, the pair of $1/2$ landing points are collapsed, as are the pair of $0$ landing points since we take the angles of external rays modulo 1. Gluing two copies of $T_{1/4}$ in this manner produces the graph on the right of Figure \ref{f14selfi}, which we take to be the 1-skeleton of the 2-sphere tiling $S_\mathcal{R}$.

\begin{figure}[htb]
\center{\includegraphics{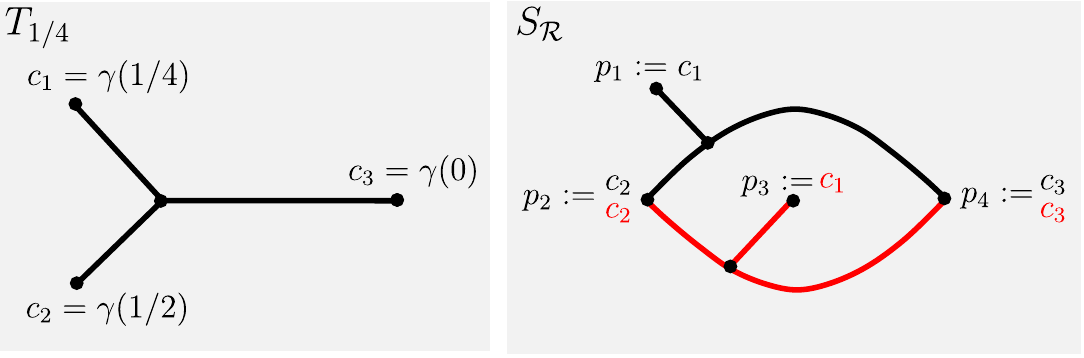}}
\caption{On the left, $T_{1/4}$. On the right, the subdivision complex $S_\mathcal{R}$ for the essential self-mating of $f_{1/4}$.}
\label{f14selfi}
\end{figure}

Now, we may pull back the 1-skeleton of $S_\mathcal{R}$ by the essential mating $g=f_{1/4}\upmodels_e f_{1/4}$, as in Figure \ref{f14selfii}. This process mostly resembles pulling back two Hubbard trees by their respective polynomials (a solved problem as noted in Figure \ref{hubbardpreim}) and keeping track of identifications between these trees by $\sim_e$ at only a finite number of points. A good way to view reconciling these identifications is that the essential mating is a branched covering of the 2-sphere, and should thus behave locally homeomorphically except on the critical set. This means that we must `fill in' new edges in a manner that preserves this homeomorphic behavior, which is up to our discretion by step four in Definition \ref{essential}. We thus obtain the subdivided tiling $\mathcal{R}(S_\mathcal{R})$. 

We can then use the boundary behavior of the four 2-tile hexagons in $\mathcal{R}(S_\mathcal{R})$ to note that these tiles map homeomorphically via $g$ onto the two open 2-tiles in $S_\mathcal{R}$. Thus, the construction generates a finite subdivision rule with subdivision map $g$.

\begin{figure}[htb]
\center{\includegraphics{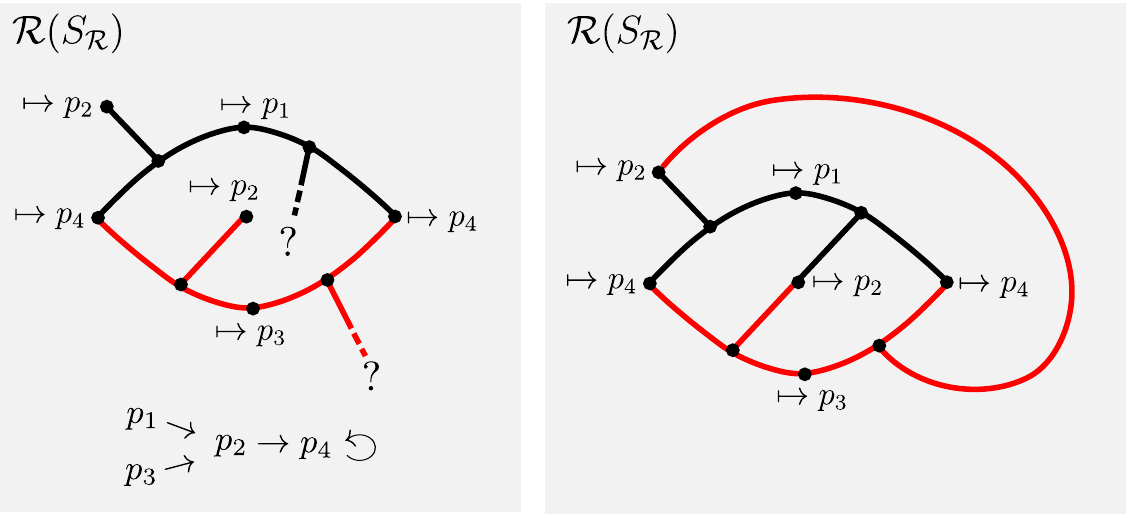}}
\caption{On the left, the expected pullback of $S_\mathcal{R}$ by the essential mating as based on local behavior of Hubbard trees. The essential mating is locally homeomorphic everywhere except on the critical set, so we complete the pullback as shown on the right.}
\label{f14selfii}
\end{figure}
\end{example}

It should be noted that there are some instances in which this construction scheme does not directly generate a finite subdivision rule: 

\begin{theorem}
Let $h$ be the formal mating of $f_\alpha$ and $f_\beta$. The essential type construction fails to yield a finite subdivision rule generated by this polynomial pairing if and only if there exists some $x,y$ in $T_\alpha\bigsqcup T_\beta$ with $x\sim_t y$, $x\not\sim_e y$, and $h(x)\sim_e h(y)$. \cite{FSRCONSTRUCTION}
\end{theorem}

In short, this theorem notes that when all open tiles of the pullback $\mathcal{R}(S_\mathcal{R})$ can be mapped onto some open tile of $S_\mathcal{R}$, we will have a subdivision rule. Otherwise, if some tile of $\mathcal{R}(S_\mathcal{R})$ maps to $S_\mathcal{R}$ with nonhomeomorphic behavior, the construction does not allow for the essential mating to serve as a subdivision map between the two tilings. In such a case, we may require slight modification to the essential type construction 
in order to obtain a formal finite subdivision rule. Options for altering the construction to obtain a valid finite subdivision rule are detailed in \cite{DISSERTATION}.

\subsection{Pseudo-equators}


Theorem \ref{LRS} expresses when a mating can be viewed as equivalent to a rational map---but what about when a rational map can be viewed as a mating? In \cite{UNMATING}, it is shown that a postcritically finite rational map with Julia set the 2-sphere can be expressed as a mating if the map possesses a \emph{pseudo-equator}:

\begin{definition}\label{pseudoequator}

A homotopy $H: X\times [0,1]\rightarrow X$ is a \emph{pseudo-isotopy} if $H: X\times [0,1)\rightarrow X$ is an isotopy. We will assume $H_0=H(x,0)=x$ for all $x\in X$.

Let $f$ be a postcritically finite rational map, $C_0\subseteq \hat{\mathbb{C}}$ be a Jordan curve with $P_f\subseteq C_0$, and $C_1=f^{-1}(C_0)$. Then we say that $f$ has a \emph{pseudo-equator} if it has a pseudo-isotopy $H: \mathbb{S}^2\times[0,1]\rightarrow \mathbb{S}^2$ rel. $P_f$ with the following properties:

\begin{enumerate}
\item $H_1(C_0)=C_1$.
\item The set of points $w\in C_0$ such that $H_1(w)\in f^{-1}(P_f)$ is finite. (We will let $W$ denote the set of all such $w$.)
\item $H_1:C_0\backslash W\rightarrow C_1\backslash f^{-1} (P_f)$ is a homeomorphism.
\item $H$ deforms $C_0$ orientation-preserving to $C_1$.
\end{enumerate}

\end{definition}

Possession of a pseudo-equator is not a necessary condition for a rational map to be equivalent to a mating, but considering how such a property arises from a mating will be useful in developing insight on the mapping properties of the essential mating.

\begin{theorem}\label{pseudothm}
Set $g=f_\alpha\upmodels_ef_\beta$, and let $\mathbb{S}'^2$ denote the quotient space which is the domain of $g$. If there exists some Jordan curve $C$ on $\mathbb{S}'^2$ which contains $P_g$ and separates $(T_\alpha/\sim_e)\setminus P_g$ from $(T_\beta/\sim_e)\setminus P_g$, then $g$ has a pseudo-equator. \cite{FSRCONSTRUCTION}
\end{theorem}

We provide a sketch of the proof in \cite{FSRCONSTRUCTION}.  Given a finite subdivision rule formed using the construction from the previous section, the curve $C_0$ and associated pseudo isotopy can sometimes be constructed in a natural way from the equator $\mathcal{E}$ of the formal mating. If $C_0:=\mathcal{E}/\sim_e$ is a Jordan curve, $C_0$ separates the 2-sphere into two components; the closure of each containing the Hubbard tree of a polynomial in the mating. We will assume one to be colored black and one to be colored red. More simply, $C_0$ will appear as a simple closed curve drawn through $P_g$ that `separates' the red and black trees in the 1-skeleton of $S_\mathcal{R}$. We can then find the pullback curve $C_1$ using the finite subdivision rule construction from the previous section: the subdivision map assists us in pulling back $C_0$ as in Figure \ref{meyerex}, since open 2-tiles map homeomorphically.

We may then construct a pseudo-isotopy as described in Definition \ref{pseudoequator} between $C_0$ and its pullback $C_1$. Since both $C_0$ and $C_1$ can be taken to separate the red and black Hubbard trees off of the postcritical set, we may assume an orientation on both of these curves given by traversing each in the direction of increasing external angles for the black polynomial. This suggests natural parameterizations $C_0, C_1:[0,1)\rightarrow \mathbb{S}^2$ where points on each curve are given as a function of the associated external angle. We may then view $H$ as any homotopy which continuously distorts $C_0$ into $C_1$ by `pushing' each point $C_0(t)$ along the $t$ external ray to the point $C_1(t)$. (We may note that since both curves pass through the postcritical set, these points will be fixed by the homotopy.) Such a homotopy preserves orientation. Further, since $g$ is the subdivision map of a finite subdivision rule, we have that the remaining conditions on finiteness and homeomorphic mapping behavior for a pseudo-equator are satisfied.

A crucial idea to observe here is that if we can build a finite subdivision rule using an essential mating which has one of these pseudo-equators, pullbacks of  $C_0$ behave in a predictable manner due to the pseudo-isotopy: they can be formed by deforming our original curve in an orientation-preserving manner and `pinching' at the critical points, as in Example \ref{f14f14pseudoex}:

\begin{example}\label{f14f14pseudoex}

We deepen our examination of the essential self-mating $g=f_{1/4}\upmodels_e f_{1/4}$. In Figure \ref{meyerex}, we first note the previously obtained finite subdivision rule for this mating: the 1-skeletons of $S_\mathcal{R}$ and $\mathcal{R}(S_\mathcal{R})$ are shown in red and black. Next, we form $C_0$ on the left by drawing a Jordan curve through $P_g$ that separates the red and black Hubbard trees on $S_\mathcal{R}$. Noting the homeomorphic mapping behavior of $g$ on open tiles of $\mathcal{R}(S_\mathcal{R})$, we may determine the location of the pullback curve $C_1$ on the right. We establish an orientation on both $C_0$ and $C_1$ to be given by traversing each curve so that the black tree is always on the left.

\begin{figure}[htb]
\center{\includegraphics{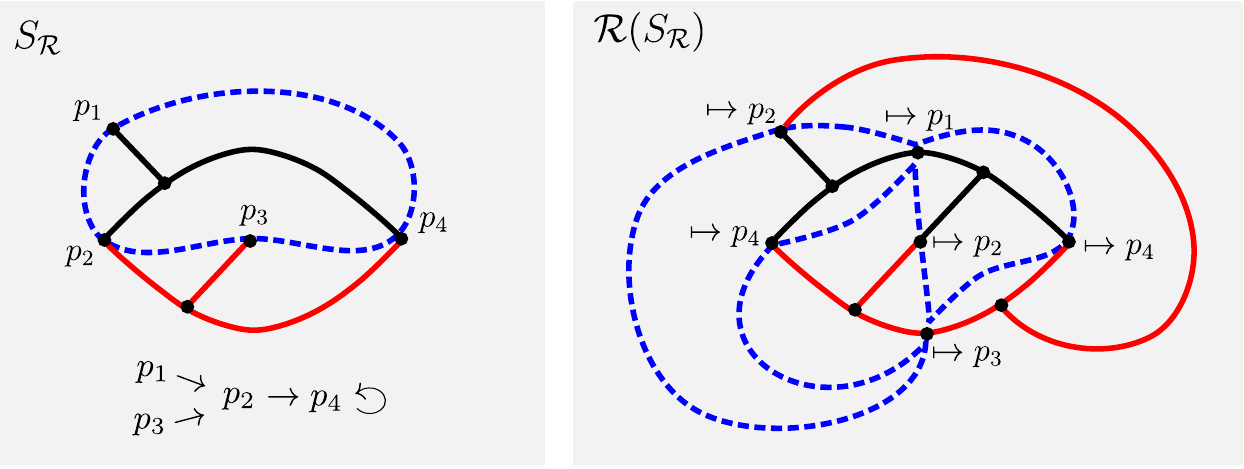}}
\caption{\label{meyerex}{The finite subdivision rule associated with $f_{1/4}\upmodels_ef_{1/4}$, along with marked pseudo-equator curves. $C_0$ is marked in blue on the left and its pullback $C_1$ is marked in blue on the right.}}
\end{figure} 

In Figure \ref{meyerex}, the action of the pseudo-isotopy $H$ is to deform $C_0$ into $C_1$ by pinching the arc from $p_1$ to $p_4$ and the arc from $p_2$ to $p_3$ together at the critical point of the black tree. Simultaneously, $H$ pinches the remaining arcs together at the critical point of the red tree.\end{example}




\section{Main Results}\label{algorithms}

\subsection{Theory for the pseudo-equator algorithm}\label{convergence}

We utilize two normalization conventions given in \cite{MEDUSA}; one regarding embeddings of $\mathbb{S}$ to $\hat{\mathbb{C}}$ and one regarding rational maps. 

First, suppose that $g:\mathbb{S}^2\rightarrow\mathbb{S}^2$ is the essential mating of two critically preperiodic quadratic polynomials $f_\alpha$ and $f_\beta$, and that $w_\alpha,w_\beta$ are the two critical points of $g$. Let $\mathcal{H}$ be the set of orientation preserving maps $\sigma:\mathbb{S}^2\rightarrow\hat{\mathbb{C}}$, normalized so that $\sigma(w_\alpha)=0, \sigma(w_\beta)=\infty,$ and $\sigma(1)=1$.

Then, any $\sigma\in\mathcal{H}$ can be taken as a global chart yielding a complex structure on $\mathbb{S}^2$. In this manner, we can take $\sigma$ to be a representative of some element of $\mathcal{T}_f$. Conversely, $\mathbb{S}^2$ equipped with a complex structure is conformally isomorphic to $\hat{\mathbb{C}}$, and so we may assume the existence of an associated conformal isomorphism $\sigma: \mathbb{S}^2\rightarrow\hat{\mathbb{C}}$ normalized on $0,1$, and $\infty$ as above.

Next, we note that rational maps of degree 2 can be normalized via conjugation with Mobius transformations so that 0 and $\infty$ are critical points and 1 is a fixed point. We'll refer to the collection of such normalized maps as $\mathcal{F}$, and note the following lemma:

\begin{lemma}[Henriksen, Lynch Boyd]\label{rationalmaplemma} Given two distinct points $u,v\in \hat{\mathbb{C}}\backslash\{1\}$, there exists a unique $F\in \mathcal{F}$ so that $F(0)=u$ and $F(\infty)=v$. If $u,v\neq\infty$, this map $F$ takes the form $F_{u,v}(z)= \frac{(u-1)vz^2-u(v-1)}{(u-1)z^2-(v-1)}$. The desired map $F$ is intuitively similar in structure in the event that $u$ or $v$ is equal to $\infty$.

Any degree 2 rational map normalized in this manner is uniquely determined by its two critical values $u$ and $v$. \cite{MEDUSA}
\end{lemma}

We may then present the following theorems:

\begin{theorem}\label{convergencetheorem}
Let $\sigma_n\in\mathcal{H}$ be given. Set $u_n=\sigma_n\circ g(w_\alpha), v_n=\sigma_n\circ g(w_\beta)$, and let $F_{u_n,v_n}$ be the map described in Lemma \ref{rationalmaplemma}. Then, there exists a unique mapping $\sigma_{n+1}\in\mathcal{H}$ such that the following diagram commutes:

$$\begin{CD}
\mathbb{S}^2		@>\sigma_{n+1}>> 		\hat{\mathbb{C}}\\
@VVgV			 					@VVF_{u_n,v_n}V\\
\mathbb{S}^2		@>\sigma_n>> 		\hat{\mathbb{C}}
\end{CD}
$$

Further, if $\sigma_n$ and $\sigma_n'$ represent the same element in $\mathcal{T}_g$, then $F_{u_n,v_n}$ and $F_{u_n,v_n}'$ are the same rational map and the lifts $\sigma_{n+1}$ and $\sigma_{n+1}'$ similarly represent the same element of $\mathcal{T}_g$ as well.
\end{theorem}

\begin{theorem}\label{end}
Fixing a starting $\sigma_0\in\mathcal{H}$ and repeatedly applying Theorem \ref{convergencetheorem} generates a sequence of rational maps $F_{u_n,v_n}$ that is equivalent to those produced by Thurston's algorithm. This sequence converges to a rational map $F$ that is Thurston-equivalent to the topological mating of $f_\alpha$ and $f_\beta$.
\end{theorem}

This collection of assertions is similar in nature to Lemma 3.7, Theorem 3.8, and Theorem 3.9 of \cite{MEDUSA}, but generalized as we are not working with elements of Medusa space. The proofs follow in a similar manner.

\begin{proof}[Proof of Theorem \ref{convergencetheorem}]

Although $g$, $\sigma_n$, and $F_{u_n,v_n}$ are maps on $\mathbb{S}^2$ and $\hat{\mathbb{C}}$, we consider the following diagram on doubly punctured spheres:

$$\begin{CD}
\mathbb{S}^2\setminus\Omega_f	@.	 				\hat{\mathbb{C}}\setminus \{0,\infty\}\\
@VVgV			 											@VVF_{u_n,v_n}V\\
\mathbb{S}^2\setminus f(\Omega_f) 		@>\sigma_n>> 			\hat{\mathbb{C}}\setminus \{u_n,v_n\}
\end{CD}
$$\\

The fundamental group of any doubly punctured sphere is $\mathbb{Z}$. More specifically, we may fix 1 as a base point and identify the fundamental group on the doubly punctured sphere $\mathbb{S}^2$ (respectively, $\hat{\mathbb{C}}$) with $\mathbb{Z}$ so that the winding number about $w_\alpha$ or $g(w_\alpha)$ (respectively, about 0 or $u_n$) corresponds to the element $+1\in\mathbb{Z}$. The maps $g$ and $F_{u_n,v_n}$ are degree 2 branched covers of the sphere, and so are two-to-one covering maps when we omit branch points and preimages as above. The induced maps on fundamental groups $g_*$ and $F_{u_n,v_n*}$ are then both equivalent to multiplication by 2.

$\sigma_{n}$ on the other hand is a homeomorphism, so the induced map $\sigma_{n*}$ is equivalent to the identity. We then have that $F_{u_n,v_n*}$ and $(\sigma_n\circ g)_*=\sigma_{n*}\circ g_*$ have the same image. By the fundamental lifting theorem, there is a lift $\sigma_{n+1}: \mathbb{S}^2\setminus\Omega_g\rightarrow \hat{\mathbb{C}}\setminus\{0,\infty\}$ that commutes with the diagram above. This map $\sigma_{n+1}$ is unique if we specify that $\sigma_n(1)=1$.  We may then continuously extend $\sigma_{n+1}$ to a map on spheres by assigning $\sigma_{n+1}(w_\alpha)=0$ and $\sigma_{n+1}(w_\beta)=\infty$ so that $\sigma_{n+1}\in\mathcal{H}$. This shows that the diagram given in the statement of Theorem \ref{convergencetheorem} commutes.

For the uniqueness portion of Theorem \ref{convergencetheorem}, we consider the following. If $\sigma_n$ and $\sigma_n'$ represent the same element in $\mathcal{H}$, there exists an isotopy relative to $P_g$ between these two maps. Since the isotopy does not disturb elements of the postcritical set, $u_n$ and $v_n$ are unchanged, thus $F_{u_n,v_n}$ and $F'_{u_n,v_n}$ are the same map. Further, our isotopy lifts to a new isotopy between $\sigma_{n+1}$ and $\sigma_{n+1}'$, and so these two lifts represent the same element in $\mathcal{T}_f$. 
\end{proof}

\begin{proof}[Proof of Theorem \ref{end}]

The reader may note that while Thurston's algorithm should use a pullback to define the rational map $F_n=\sigma_n\circ g\circ \sigma_{n+1}\ ^{-1}$, the above proof defines $\sigma_{n+1}$ as a lift, assuming that the analogous rational map $F_{u_n,v_n}$ is known. A useful consequence of working in normalized maps from $\mathcal{H}$ and $\mathcal{F}$ is that we can note both $F_n$ and $F_{u_n,v_n}$ always map $0\mapsto u_n$, $\infty\mapsto v_n$, and $1\mapsto 1$; which uniquely determines $F_n=F_{u_n,v_n}$ as a single member of $\mathcal{F}$. Thus, once we know $\sigma_{n-1}$ (and so the values of $u_n$ and $v_n$), we know $F_n$. Theorem \ref{convergencetheorem} guarantees that the lift map $\sigma_{n+1}$ is unique, and so the $\sigma_n$ in our algorithm and Thurston's algorithm coincide. We can then view the repeated application of Theorem \ref{convergencetheorem} as an algorithm generating the same sequence of embeddings $\sigma_n$ and rational maps $F_{u_n,v_n}$ as Thurston's algorithm. 

Since the essential mating $g$ is Thurston-equivalent to the topological mating of $f_\alpha$ and $f_\beta$, we conclude by Thurston's algorithm that the sequence of rational maps $F_{u_n,v_n}$ converges to a rational map Thurston-equivalent to $g$. 
\end{proof}

\subsection{Implementation of the pseudo-equator algorithm} 

While Theorem \ref{end} makes obtaining the rational map $F$ appear easy, implementation of the theorem as a computer algorithm involves some attention to detail. Our key result, an algorithm that obtains an approximation for the geometric mating of two polynomials, will be organized into a process that involves five major steps. We will call this the \emph{pseudo-equator algorithm}:

\begin{algorithm}\label{algorithmthing} Suppose that essential mating of two critically preperiodic polynomials has a hyperbolic orbifold, and that this mating permits construction of a finite subdivision rule and pseudo-equator curve. The \emph{pseudo-equator algorithm} refers to the following process for approximating the geometric mating for these two polynomials, as described below:

\begin{enumerate}
\item Build the finite subdivision rule for the essential mating, per Definition \ref{fsrdefn}.
\item Construct a pseudo-equator and embedding, per Theorem \ref{pseudothm}.
\item Assign an approximation for the rational map, per Lemma \ref{rationalmaplemma}.
\item Pull back the curve while noting locations of preimages of marked points, per Theorem \ref{convergencetheorem}.
\item Repeat the approximation and pullback steps to obtain a sequence of rational maps, per Theorem \ref{end}.
\end{enumerate}

\end{algorithm}

It should be recalled that Theorem \ref{end} guarantees not just the existence of some sequence of rational maps, but that this sequence converges to a desired rational map $F$ that can be taken as the geometric mating of our two polynomials. We expand upon each of these steps and their roles within the algorithm below:\\

\noindent \textbf{(1) Build the finite subdivision rule for the essential mating.}
	
	There are several finite subdivision rule constructions available to describe the action of matings on $\mathbb{S}^2$ so we can apply Thurston's algorithm. In the event that the essential and formal matings are equivalent, we may use formal mating constructions given in \cite{DISSERTATION}; otherwise, we use constructions detailed in Section \ref{yourfsrrules} and in  \cite{FSRCONSTRUCTION}. Since the Medusa algorithm applies in the former case , we focus our efforts on understanding the latter situation.
	
	When $f_\alpha$ and $f_\beta$ are not strongly mateable, the essential type finite subdivision rule construction involves identifying Hubbard trees at marked points specified by $\sim_e$. This yields a 1-skeleton that can be completed to a tiling  $S_\mathcal{R}$ on $\mathbb{S}^2$. Off of the marked points, the action of $g=f_\alpha\upmodels_e f_\beta$ on the 1-skeleton is similar to the action of $f_\alpha$ or $f_\beta$ on its associated Hubbard tree--that is, we have homeomorphic behavior off of the critical set. If we note expected behavior of marked points under the essential mating, we may develop a new 1-skeleton that can be completed to a subdivided tiling $\mathcal{R}(S_\mathcal{R})$, with the essential mating acting as a subdivision map.

\noindent\textbf{(2) Construct a pseudo-equator and embedding. }
	
	If we have a finite subdivision rule that was generated in the above manner, the easiest way to attempt to find a pseudo-equator is to construct a closed curve $C_0$ through $P_g$ that separates the two Hubbard trees in the tiling 1-skeleton off of $P_g$. If this is a Jordan curve, Theorem \ref{pseudothm} guarantees that $g$ has a pseudo-equator. We may then use the finite subdivision rule to determine the pseudo-equator curve's pullback $C_1$ under the essential mating, since $n$-tiles map homeomorphically to other $n$-tiles. This pullback will then have a pseudo-isotopy $H_1:\mathbb{S}^2\times[0,1]\rightarrow \mathbb{S}^2$ so that $H_1(\cdot,1)$ maps $C_0$ orientation preserving to $C_1$. 
	
	To embed $C_0$ in $\hat{\mathbb{C}}$, we may without loss of generality select the black polynomial $f_\alpha$ to fix an orientation of the curve: we will assume $C_0$ to be the positively oriented boundary curve around the component of $\mathbb{S}^2\setminus C_0$ containing the black critical point. Recall that the polynomials $f_\alpha$ and $f_\beta$ have postcritical points given by landing points of external rays $\gamma_\alpha(\alpha\cdot 2^{n-1})$ and $\gamma_\beta(\beta\cdot 2^{n-1}), n\in\mathbb{N}$. Further, considering the external angles associated with $f_\alpha$, we view the curve $C_0$ as a path $C_0:[0,1]\rightarrow\mathbb{S}^2$ possessing a natural parameterization $C_0(t)$. (We may do this, in fact, for all pullbacks of $C_0$ as well.) Motivated by this, we let $\sigma_0:\mathcal{S}^2\rightarrow\hat{\mathbb{C}}$ be the map such that $C_0(t)\mapsto e^{2\pi i t}$, and on this unit circle we will mark the points 1, $\{e^{2\pi i \alpha \cdot 2^{n-1}}\}$ and $\{e^{2\pi i (1-\beta) \cdot 2^{n-1}}\}$, $n\in \mathbb{N}$ to correspond to the fixed point and postcritical points of the essential mating. 
		
	We complete the definition of $\sigma_0$ and extend it to an orientation preserving complex structure from $\mathbb{S}^2\rightarrow\hat{\mathbb{C}}$ by defining $\sigma_0$ to be a homeomorphic extension sending $w_\alpha$ to 0, $w_\beta$ to $\infty$, and 1 to 1. Our intent is to select $\sigma_0$ as a normalized representative of some element of Teichmuller space. While a different homeomorphic extension would in general yield a different representative of the same element in $\mathcal{T}_g$---and while we could make a similar comment regarding the exact path and parameterization for $C_0$---this is a moot point by Theorem \ref{convergencetheorem}.  The remainder of the algorithm only deals in computations regarding pullbacks and embeddings of $C_0$---and only as far as determining  the embedding of $P_g$ in $\hat{\mathbb{C}}$, the order in which the marked points of $P_g$ and their embeddings are connected, and the general homotopy type of the connecting curves.
	
	The curve $C_0$ has an clear relationship to the Medusa described in Section \ref{thurstonmedusa}, and thus the above choice of $\sigma_0$ is intuitive. The Medusa resembles an equator on $\mathbb{S}^2$, with external ray limbs reaching toward the postcritical set. In both the Medusa setting and here, these equator-like curves are initially embedded in $\hat{\mathbb{C}}$ as the unit circle. A key difference is the following: if two points identify under $\sim_e$, these points are distinct on the Medusa, and there is an external ray pair joined at the equator that connects them both. In the essential mating, this pairing of external rays has been collapsed into the single marked point it intersects on the equator.  One could say in this manner that our combinatorial model resembles a ``headband" of sorts for the Medusa model: there is a clear deformation retract from the embedding of the Medusa to the circle $\sigma_0(C_0)$.

\noindent\textbf{(3) Assign an approximation for the rational map.}

	Since postcritical points of the map $g:\mathbb{S}^2\rightarrow\mathbb{S}^2$ are marked points on $C_0$, critical values of $g$ are explicitly embedded by $\sigma_{n-1}$ as some $u_n,v_n\in\hat{\mathbb{C}}$. We may then use the embedding of these critical values to determine the map $F_n=F_{u_n,v_n}$ as in Lemma \ref{rationalmaplemma}.
		
\noindent\textbf{(4) Pull back the curve, noting locations of preimages of marked points.}
	
	Since a finite subdivision rule has been determined, $C_{n}$ and $C_{n+1}$ are ascertained by noting homeomorphic behavior on tiles, much as in step 2.  Both of these curves contain the postcritical set, since $C_{n}\supseteq P_g$ implies $C_{n+1}=g^{-1}(C_{n})\supseteq g^{-1} (P_g)\supseteq P_g$---but $C_{n+1}$ will not be a Jordan curve since it contains $g^{-1}(P_g)$, and thus the critical points of the function $g$. Since there exists a pseudo isotopy $H_n:\mathbb{S}^2\times[0,1]\rightarrow\mathbb{S}^2$ such that $H_n(\cdot,1)$ maps $C_{n}$ orientation preserving to $C_{n+1}$, we do have that $H_n$ gives a canonical manner in which to traverse $C_{n+1}$ so that we visit the points of $P_g$ in the same order as $C_{n}$. 
	
	$\sigma_{n+1}$ and $\sigma_{n}$ are both orientation preserving isomorphisms, so $H_n$ may be lifted to a pseudo-isotopy on the Riemann sphere. Thus, we may expect $\sigma_{n}(C_{n})$ and $\sigma_{n+1}(C_{n+1})$ to visit embedded points corresponding to elements of $P_g$ in an intuitively similar order as well. At this point, we establish which marked points are `necessary': we primarily care about the embedding of $P_g$, not $g^{-1}(P_g)$. We can determine the cyclic order of `important' versus `unimportant' (i.e. in $P_g$ versus not in $P_g$) marked points on $C_{n+1}$, and note that $\sigma_{n+1}$ will preserve this ordering---telling us where to embed elements of $P_g$. (We touch on further subtle nuances of this process in Section \ref{implementation}.)

\noindent\textbf{(5) Repeat the approximation and pullback steps.}
	
	The $F_n$'s give a sequence of approximations to a rational map that is Thurston-equivalent to the mating $g$.

\begin{example}\label{example1} We examine the example detailed by Milnor in \cite{MILNORMATINGS}, the self-mating of $f_{1/4}$. The astute reader will note that this map actually has a parabolic $\{2,2,2,2\}$ orbifold rather than a hyperbolic orbifold, and thus Thurston's algorithm does not actually apply---but this mating provides a simplified introduction to the algorithm, and an interesting outcome nonetheless.

\noindent\textbf{Build finite subdivision rule}: The Hubbard tree for $f_{1/4}$ appears on the left of Figure \ref{f14selfi} with the postcritical set marked. Postcritical points identify under $\sim_e$ by a $\theta$ and $1-\theta$ external angle relation, so for the self-mating we obtain the subdivision complex $S_\mathcal{R}$ shown on the right of Figure \ref{f14selfi}. The preimage of this structure under the essential mating appears as on the right of Figure \ref{f14selfii}. The tiling $S_\mathcal{R}$, subdivided tiling $\mathcal{R}(S_\mathcal{R})$, and essential mating $g=f_{1/4}\upmodels_ef_{1/4}$ form an essential type finite subdivision rule. 
	
\noindent\textbf{Construct pseudo-equator}: The desired pseudo-equator curve $C_0$ and associated pullback curve $C_1$ are shown in Figure \ref{meyerex}. If we positively orient $C_0$ with respect to the black polynomial, we may note that it passes through the marked postcritical points

\begin{center}$\{ p_1=C_0(\frac{1}{4}), p_2=C_0(\frac{1}{2}), p_3=C_0(\frac{3}{4}),p_4=C_0(0)\}$\end{center}

\noindent in the listed order. We select an embedding of $C_0$ into $\hat{\mathbb{C}}$ that sends $C_0$ to the unit circle via the mapping $C_0(t)\mapsto e^{2\pi i t}$. The above list of marked postcritical points then maps respectively to $\{ i, -1, -i,1\}$. 

\noindent\textbf{Assign rational map}: Recall that the critical values of $g$ are $p_1$ and $p_3$. We may then set $u_0=\sigma_0(p_1)$ and $v_0=\sigma_0(p_3)$. Here, since $u_0=i$ and $v_0=-i$, we may utilize Lemma \ref{rationalmaplemma} to obtain that $F_0(z)=F_{u_0,v_0}=  \displaystyle\frac{(i+1)z^2+(i-1)}{(i-1)z^2+(i+1)}$.

\noindent\textbf{Pullback}: In this mating, the ramification portrait is $p_1, p_3 \mapsto p_2 \mapsto p_4 \mapsto p_4$. The pullback of $C_0$ by $g$ traverses marked points in the following ordering: 

\begin{center}$\{C_1(\frac{1}{8}),p_1=C_1(\frac{1}{4}),C_1(\frac{3}{8}),p_2=C_1(\frac{1}{2}),C_1(\frac{5}{8}),p_3=C_1(\frac{3}{4}),C_1(\frac{7}{8}),p_4=C_1(0)\},$\end{center}

\noindent where $C_1(\frac{1}{8})=C_1(\frac{5}{8})$ and $C_1(\frac{3}{8})=C_1(\frac{7}{8})$ are the critical points of the mating $g$. (We should note that the curve $C_1$ forks right whenever it approaches the black critical point $C_1(\frac{1}{8})=C_1(\frac{5}{8})$, and left whenever it approaches the red critical point $C_1(\frac{3}{8})=C_1(\frac{7}{8})$.) Pulling back the curve $\sigma_0(C_0)$ by $F_0$ yields a curve which traverses marked points in the ordering $\{0,i,\infty,-1,0,-i,\infty,1\}$. These lists of marked points on $C_1$ and the pullback of $\sigma_0(C_0)$ induce a `new' embedding of $P_g$, which we will denote $\sigma_1$. 

\noindent\textbf{Repeat}: For this step we will assign a new rational map and repeat the pullback step. We assign the new rational map by examining the parameters $u_1=\sigma_1(p_1)$ and $v_1(\sigma_1(p_3)$, but here it so happens that $\sigma_0$ and $\sigma_1$ agree on $P_g$. (This does not typically happen in usual examples---generally, we would expect $\sigma_1$ to assign new image elements to $P_g$.) Since the critical values for $g$ are assigned to the same respective elements of $\hat{\mathbb{C}}$, $F_1=F_0$.

Since $\sigma_0$ and $\sigma_1$ are the same map, we can infer that the pullback process does not change the embedding of $P_g$. This means that $\sigma_n$ and thus $F_n$ will both be constant sequences. This means that we must have started with a representative for the fixed point of $\Sigma_g$ in Teichmuller space, and so $F(z)= \frac{(i+1)z^2+(i-1)}{(i-1)z^2+(i+1)}$ is the rational map that is Thurston-equivalent to $g$. 

\begin{figure}\label{iterates}
\includegraphics[width=1.5in]{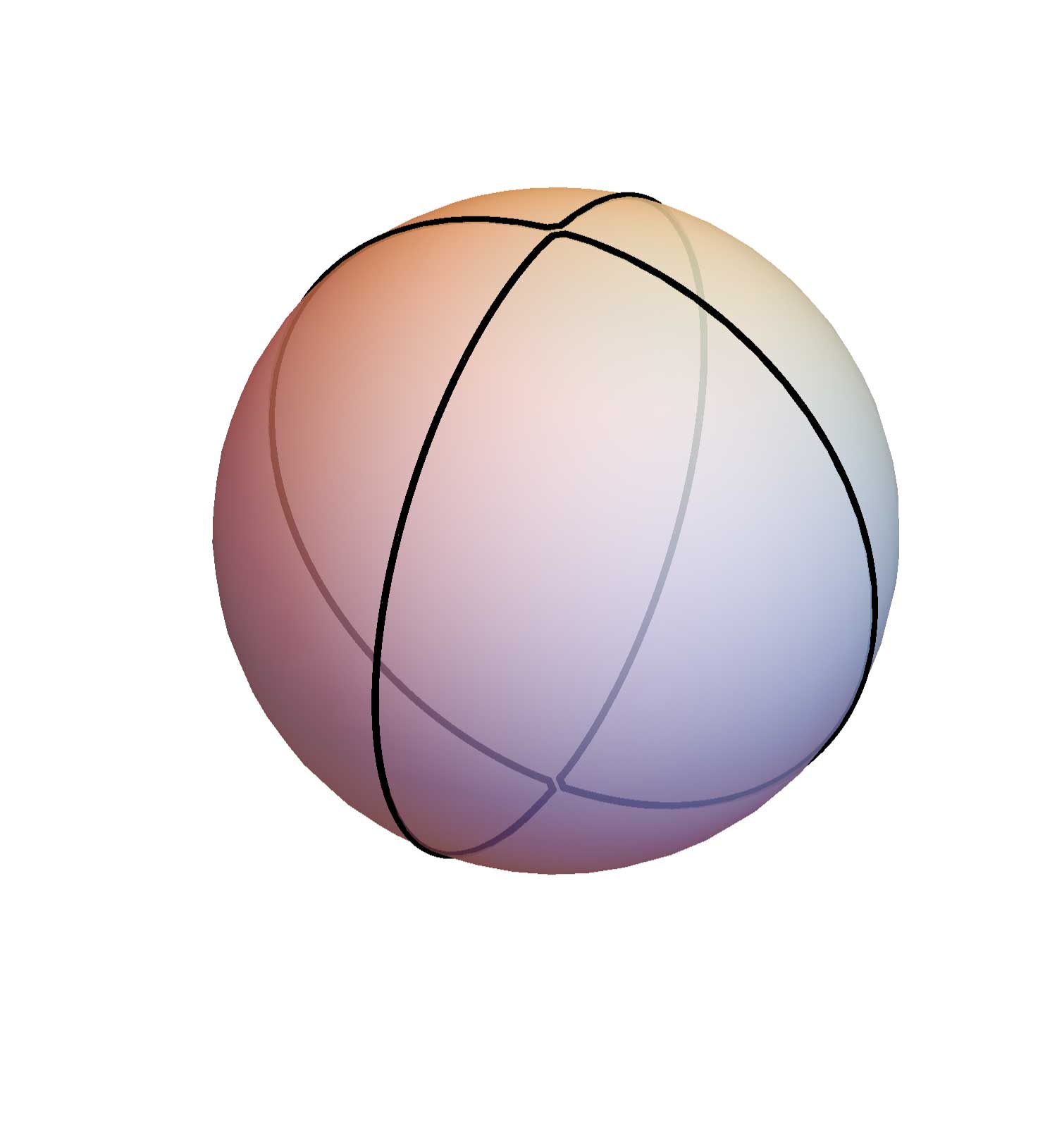}\includegraphics[width=1.5in]{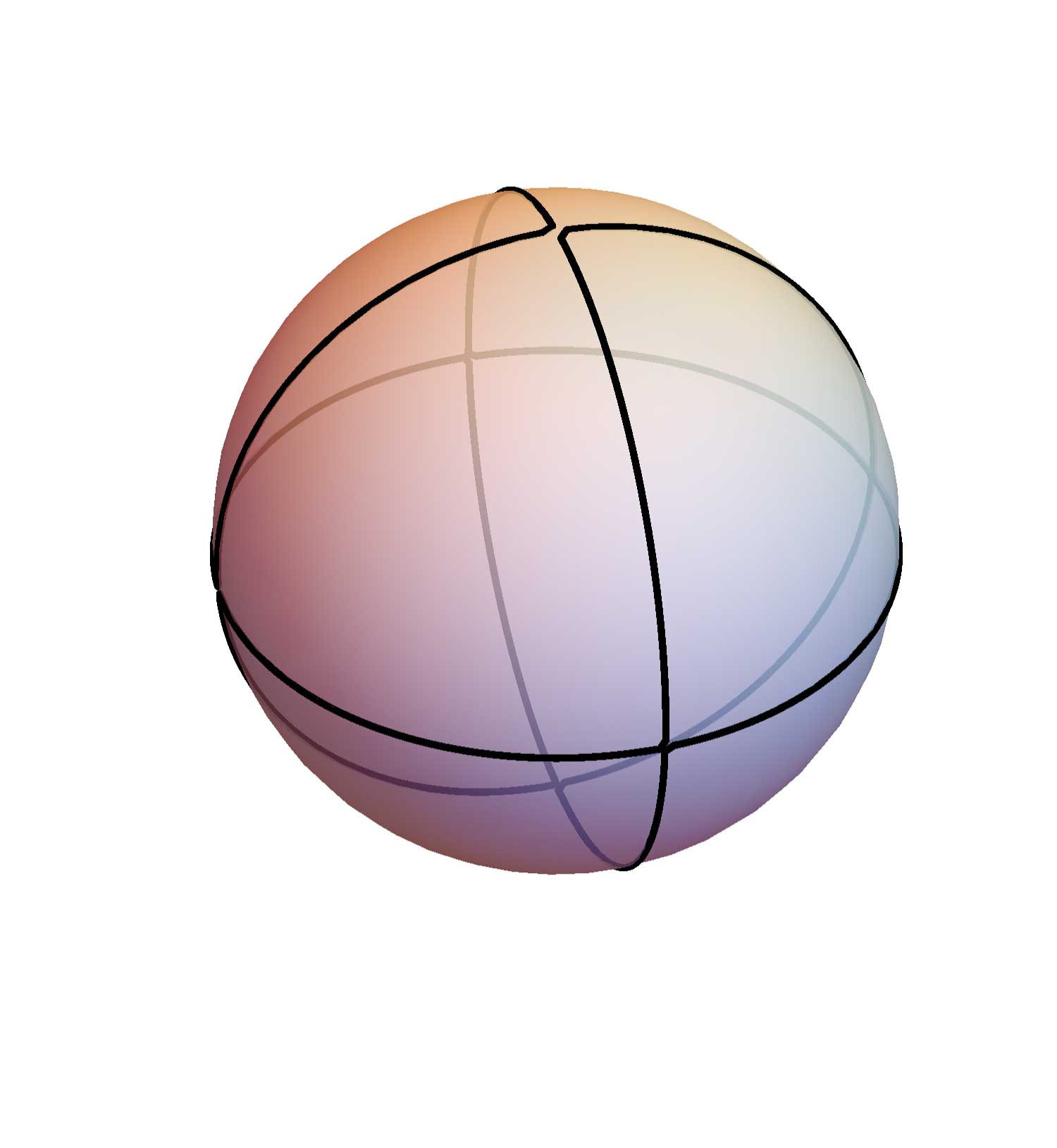}\includegraphics[width=1.5in]{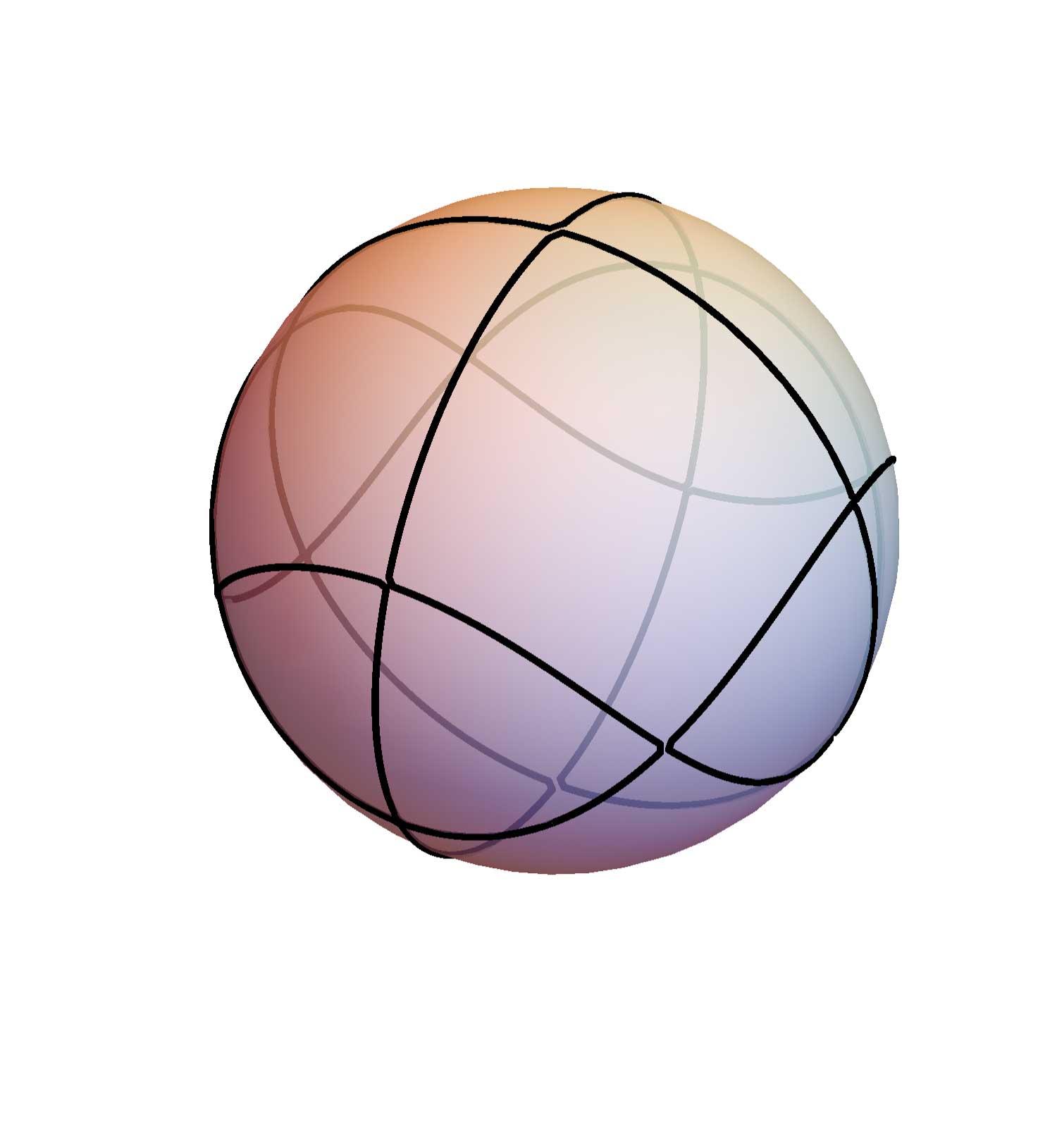}\includegraphics[width=1.5in]{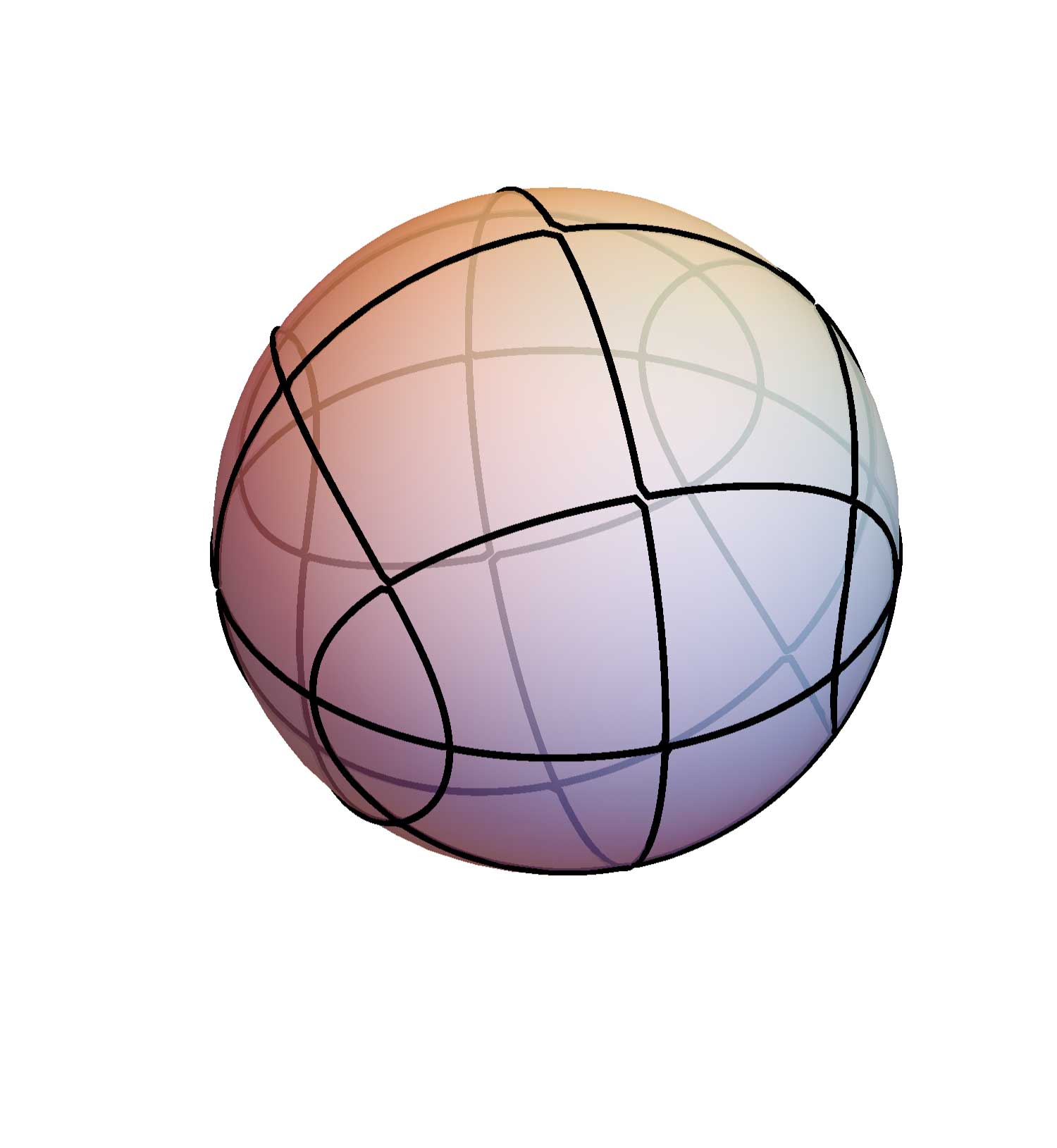}

\includegraphics[width=1.5in]{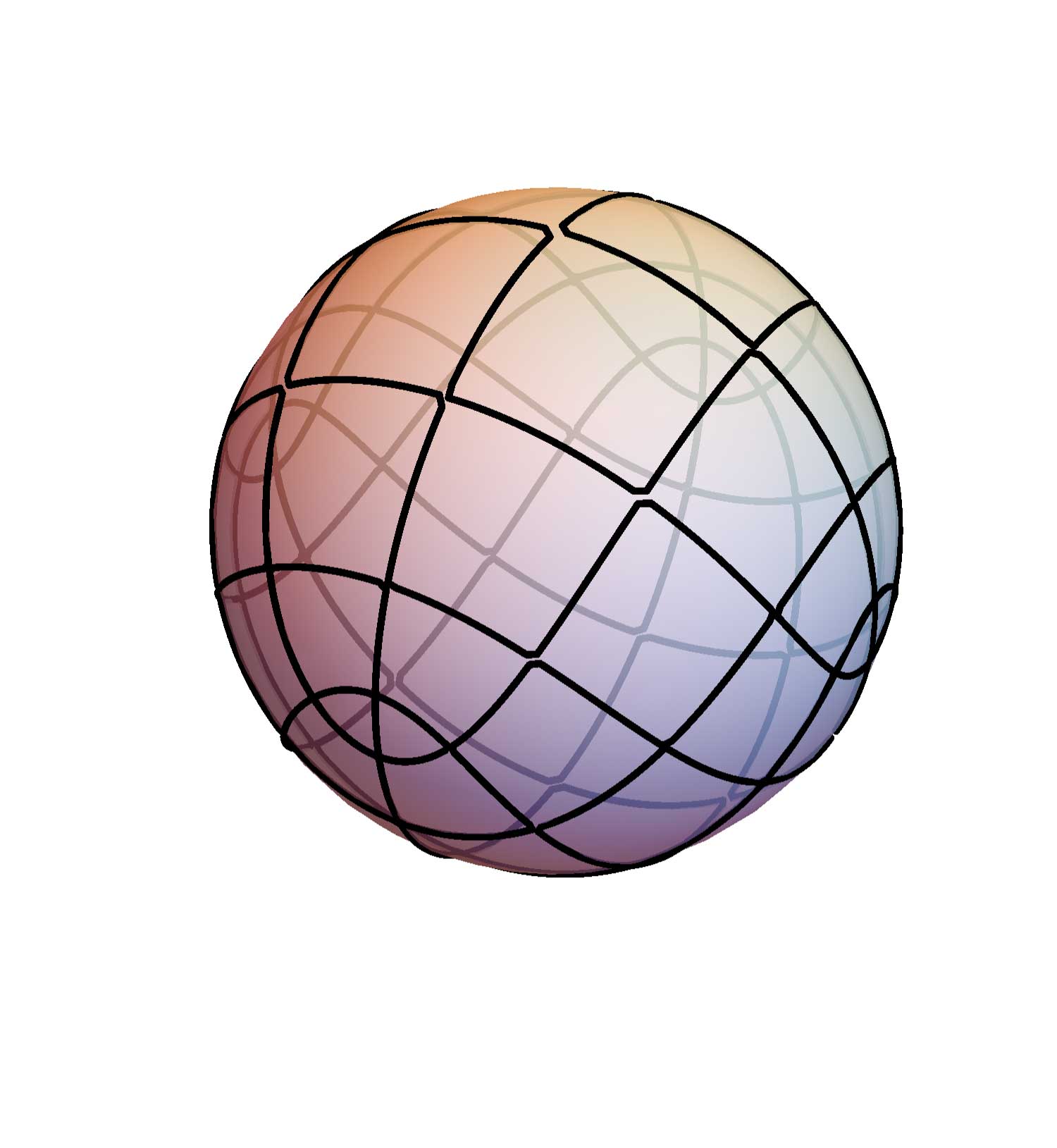}\includegraphics[width=1.5in]{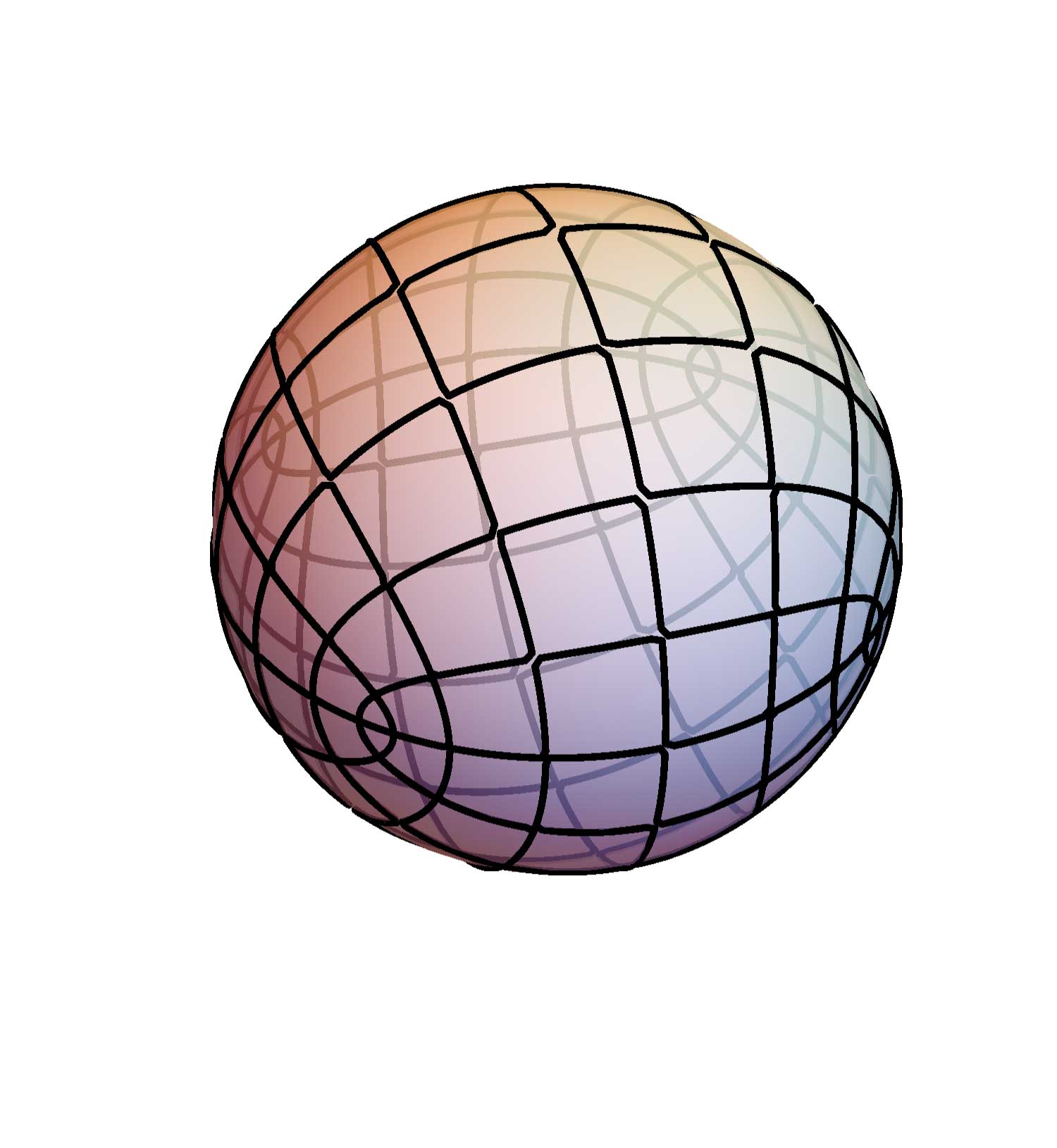}\includegraphics[width=1.5in]{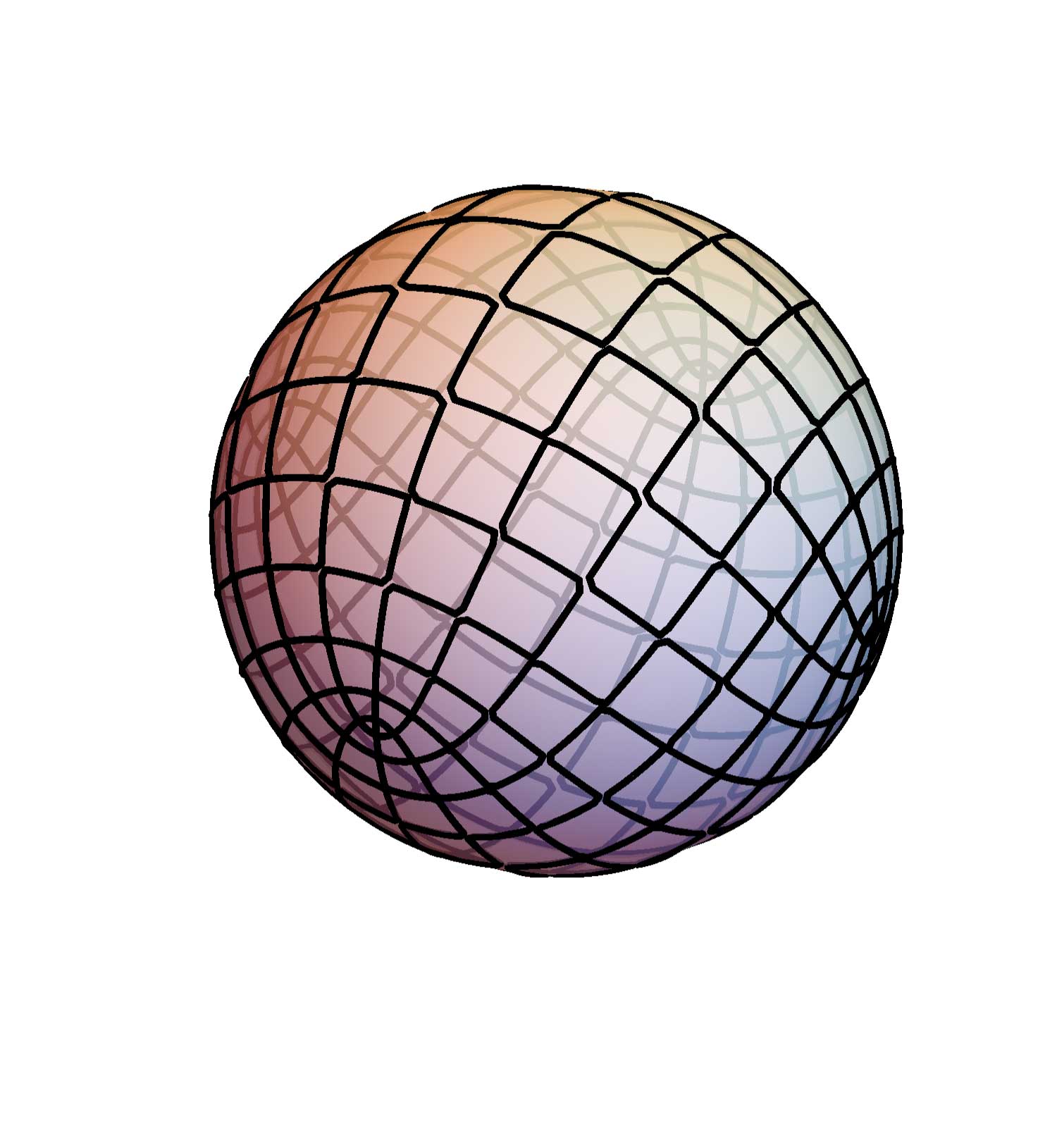}\includegraphics[width=1.5in]{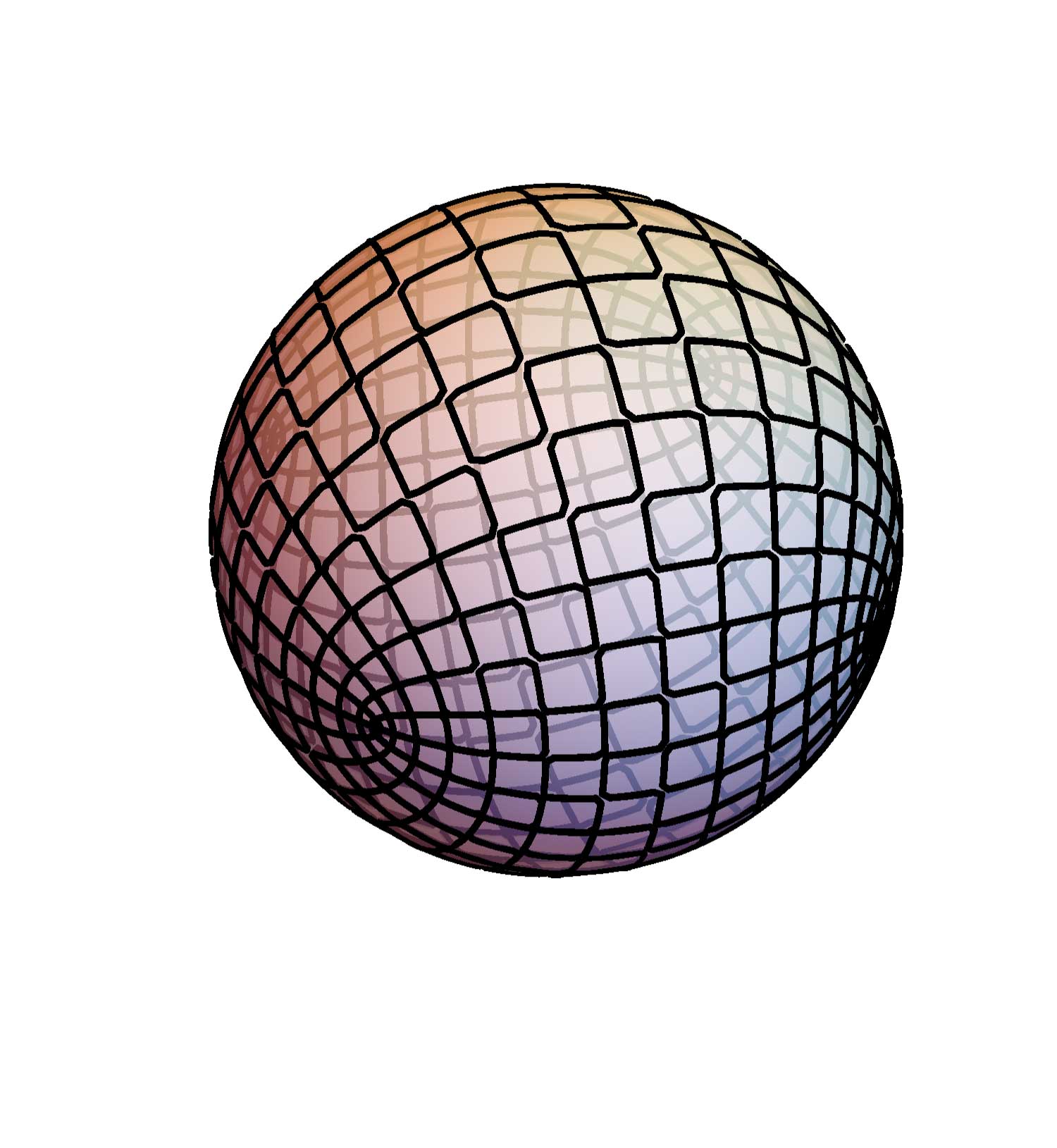}

\includegraphics[width=1.5in]{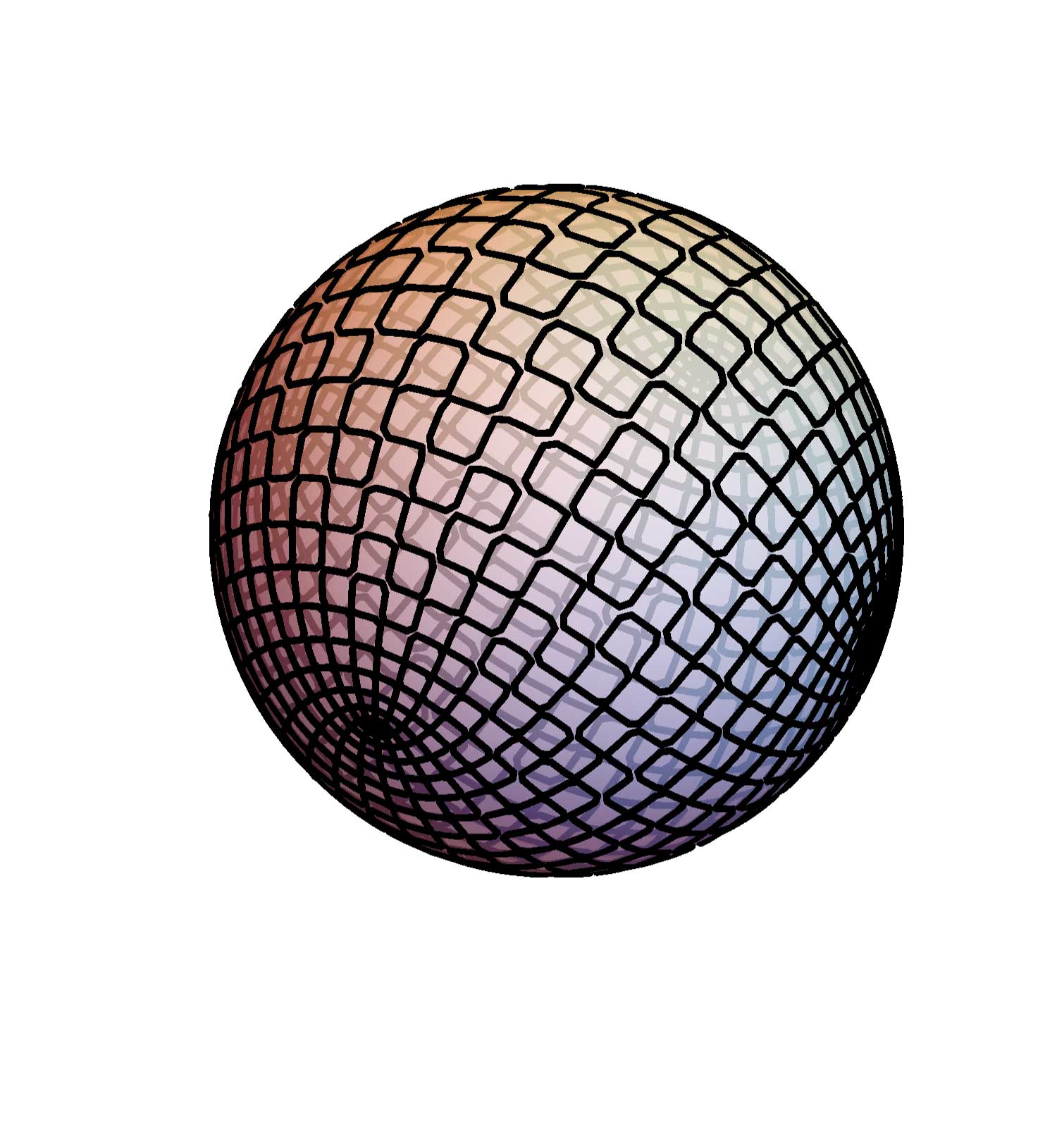}\includegraphics[width=1.5in]{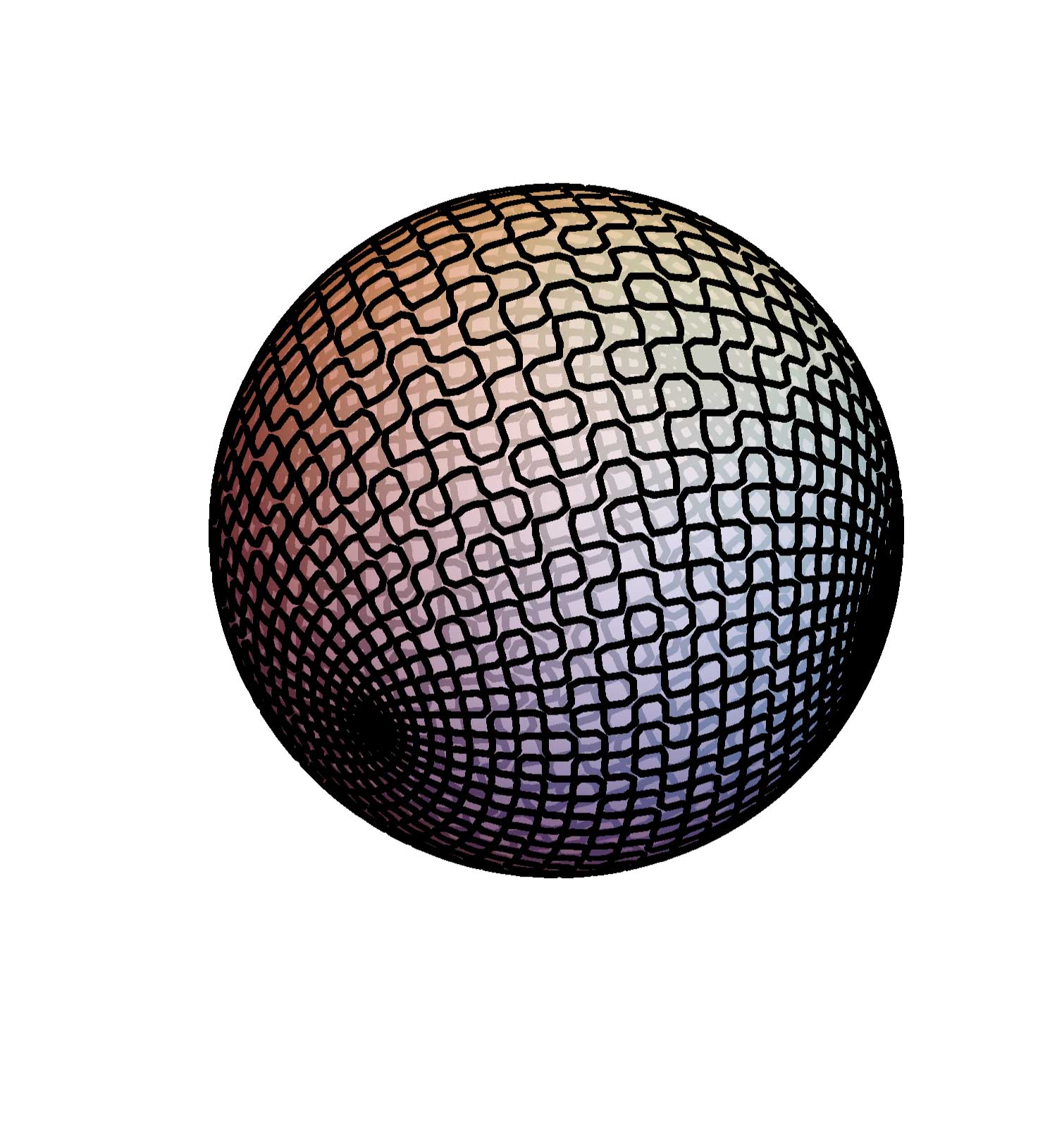}\includegraphics[width=1.5in]{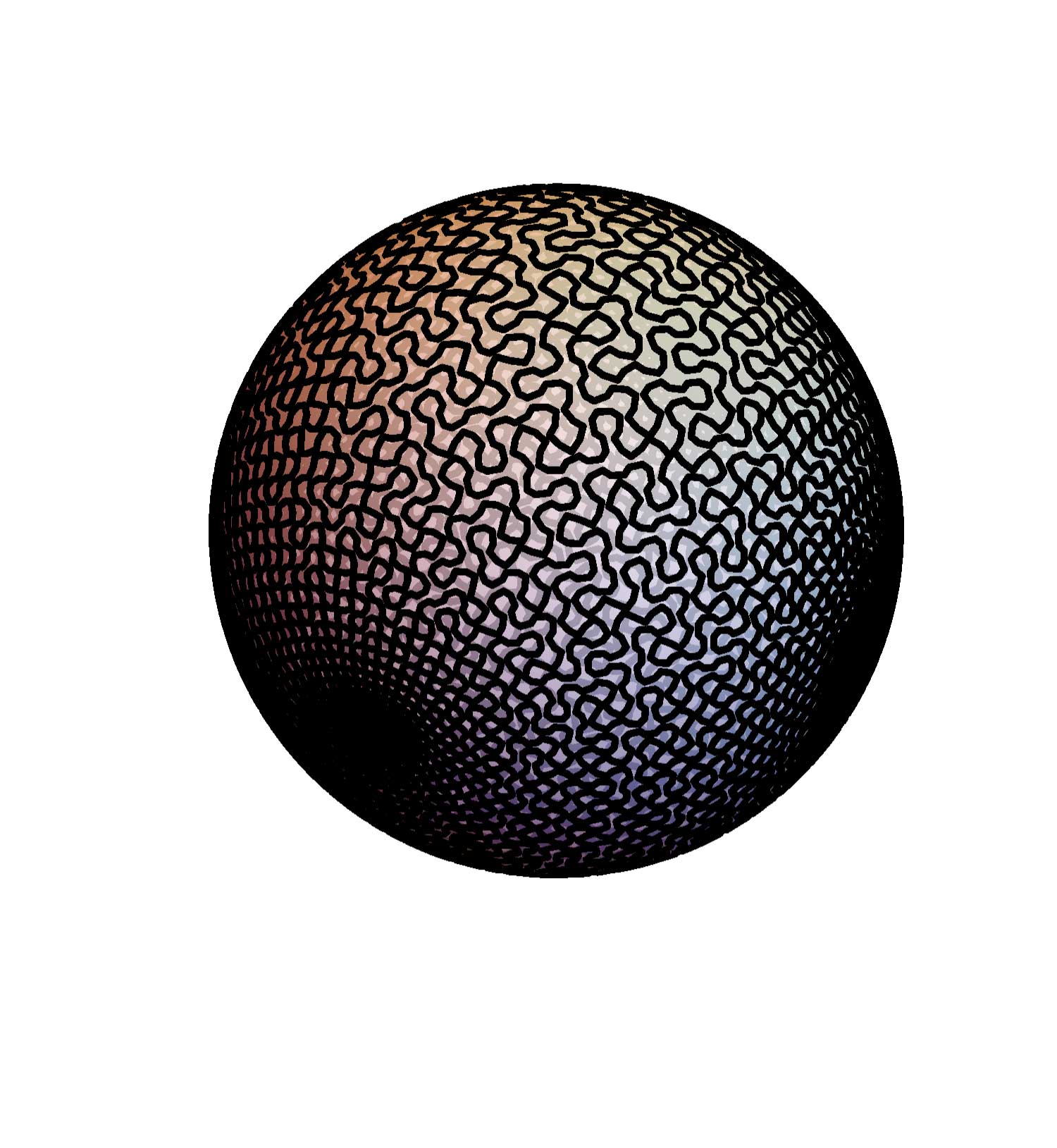}\includegraphics[width=1.5in]{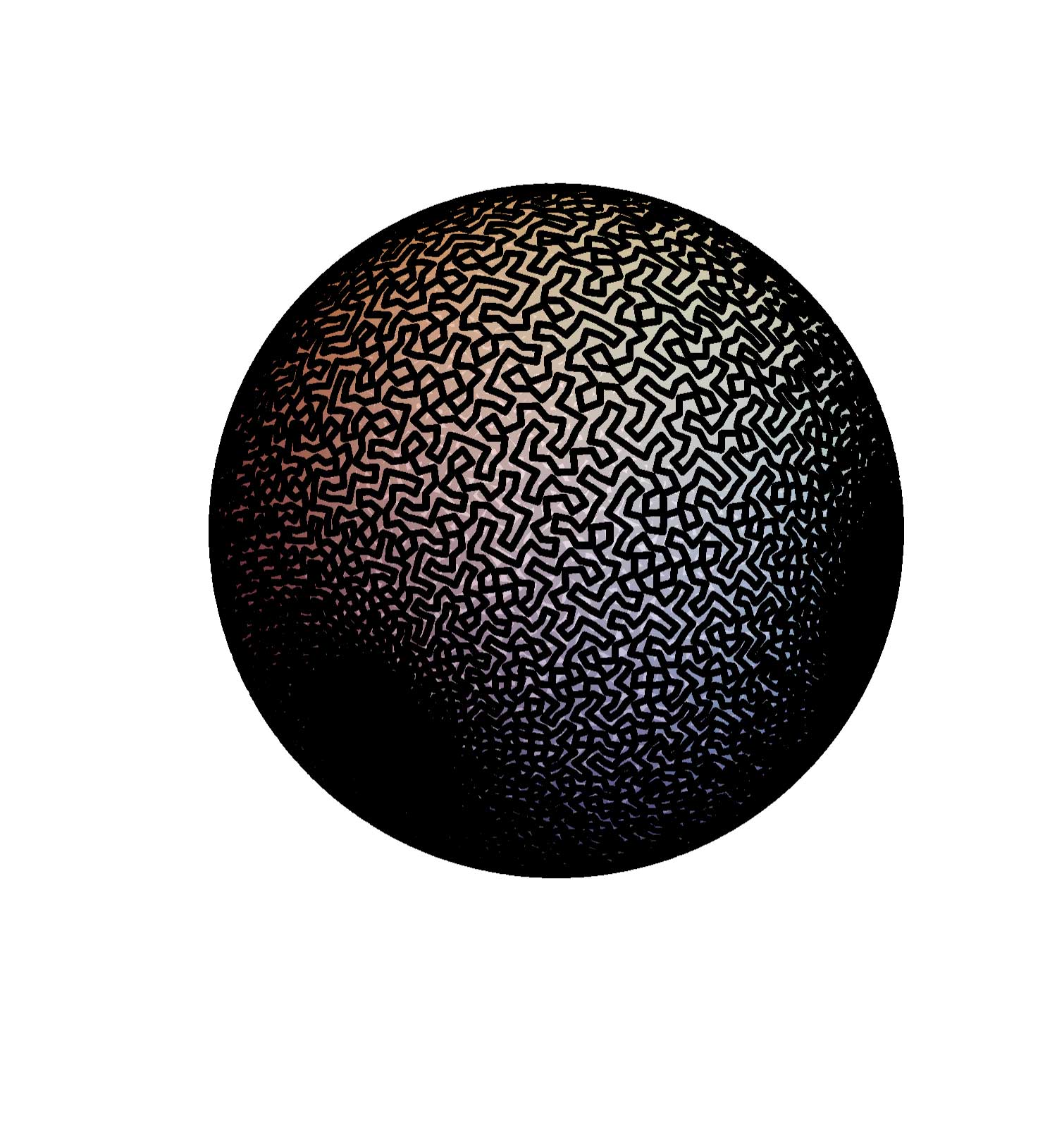}

\caption{Pullbacks of the equator by a rational map that is Thurston-equivalent to the topological self-mating of $f_{1/4}$. These pullbacks approximate the Julia set of the rational map, $\hat{\mathbb{C}}$.}
\end{figure}

The spheres in Figure \ref{iterates} demonstrate the iterated pullbacks of $\sigma_0(C_0)$ by $F$ under this process. The critical points $0$ and $\infty$ are located at the north and south poles of each of these spheres, which are depicted as slightly translucent so we may view the paths of the curves on the far side. To provide further orientation on each of the spheres, 1 is to the right, $-1$ is to the left, $i$ is to the rear right, and $-i$ is to the front left. It should be noted that the square grid formation occurring at the two critical points maps $2-1$ onto the pair of adjacent `triangular' tiles centered at $i$ and $-i$ from the previous sphere. (We say `triangular', since there is technically a marked point at $\pm i$ in the middle segment which makes these formations composed of two adjacent topological quadrilaterals.) These `triangular' tile pairs typically make the embeddings of critical values of $g$ visually appear to be a source of dark spots, or curve bunching in later iterations.  Thus, we have a loose way to visually reaffirm that the rational maps $F_n$ generated a constant sequence: the dark spots on the spheres do not change location from one iteration to the next, so the parameters $u_n$ and $v_n$ do not change, and we thus have a constant sequence for $F_n$.

\end{example}

It should be noted that generating an immediate fixed point of $\Sigma_g$ as above is actually a great stroke of luck. In general, the algorithm does not converge with other parabolic examples (such as the self-mating of $f_{1/6}$), and further it is atypical to obtain constant sequences for $\sigma_n$ when examining matings with hyperbolic orbifold. In general, the algorithm is intended to be applied to maps with five or more postcritical points (guaranteeing hyperbolic orbifold), and in such a case $F_n$ will converge non-trivially to the desired rational map. We will now demonstrate a more typical example.

\begin{example}\label{labelforexample} We examine the matings of $f_{1/4}$ and $f_{1/8}$. The essential mating has five postcritical points, thus a hyperbolic orbifold. This means that Thurston's algorithm applies to the example, and we may attempt to use Algorithm \ref{algorithmthing} to find a rational map approximation to the geometric mating of this pair.

\noindent\textbf{Build finite subdivision rule}: In Figure \ref{endexample}, we have constructed a finite subdivision rule for this mating using the construction detailed in Section \ref{yourfsrrules}. The 1-skeleton of the subdivision complex $S_\mathcal{R}$ is in black and red on the left, and the 1-skeleton of the subdivided complex $\mathcal{R}(S_\mathcal{R})$ is similarly shown in black and red on the right. With the tiling $S_\mathcal{R}$, the subdivision $\mathcal{R}(S_\mathcal{R})$, and the subdivision map $g=f_{1/4}\upmodels_ef_{1/8}$, we have a finite subdivision rule. A critical portrait for this subdivision map is shown on the left of Figure \ref{endexample}.

\begin{figure}\label{endexample}
\includegraphics{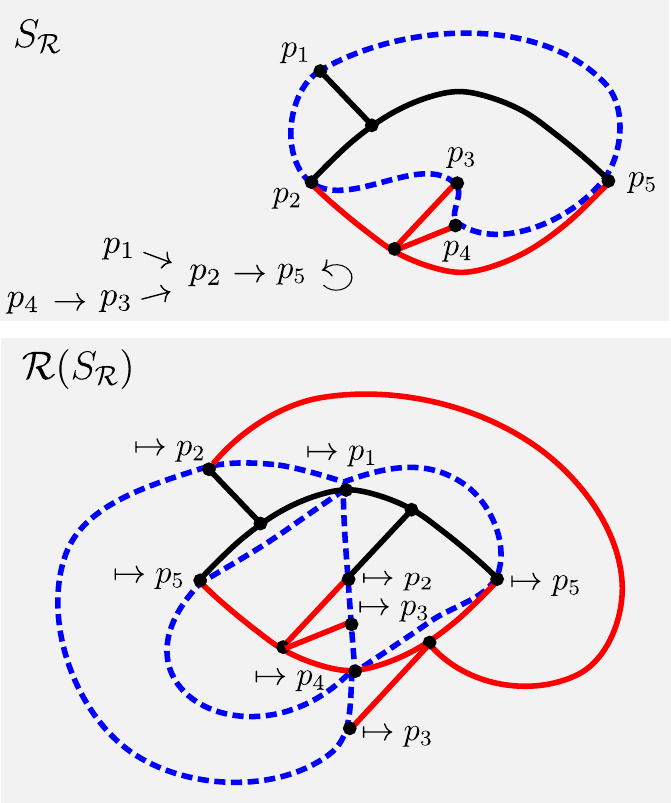}
\caption{The critical orbit portrait and finite subdivision rule associated with $f_{1/4}\upmodels_ef_{1/8}$, along with marked pseudo-equator curves. $C_0$ is marked in blue above and its pullback $C_1$ is marked in blue below. Here, $f_{1/4}$ is taken to be the black polynomial.  }
\end{figure}

\noindent\textbf{Construct pseudo-equator}: The desired pseudo-equator curve $C_0$ is constructed by separating the black $T_{1/4}$ tree from the red $T_{1/8}$ tree with a Jordan curve through the postcritical set of $g$. This curve $C_0$ is depicted in dashed blue lines on the left of Figure \ref{endexample}. Since the subdivision map $g$ is locally homeomorphic on open tiles, we infer from the finite subdivision rule that its pullback of $C_0$, the curve $C_1$, appears as the dashed blue curve on the right of Figure \ref{endexample}. 

We will orient both $C_0$ and $C_1$ positively with respect to the black polynomial. We will further assume the existence of canonical parameterizations $C_0, C_1:[0,1]\rightarrow\mathbb{S}^2$ respecting this orientation, where location on the curve is taken to be a function of the angle of external ray. This sets $C_0$ as passing through the following marked points:

\begin{center}$\{p_1=C_0(\frac{1}{4}), p_2=C_0(\frac{1}{2}), p_3=C_0(\frac{3}{4}), p_4=C_0(\frac{7}{8}), p_5=C_0(0)  \}.$\end{center}

We embed $C_0$ in $\hat{\mathbb{C}}$ as the unit circle via $\sigma_0$, which is given by the mapping $C_0(t)\mapsto e^{2\pi i t}$ for $t\in [0,1]$. Thus, the above listed marked points of $P_g$ now map respectively to the points

\begin{center}$\{ i, -1, -i, \frac{\sqrt{2}-\sqrt{2}i}{2}, 0  \}$\end{center}.

\noindent\textbf{Assign rational map}: Since $p_1$ and $p_4$ are the critical values of $g$, we will set $u_0=\sigma_0(p_1)$, and $v_0=\sigma_0(p_4)$ so that $u_0=i$ and $v_0= \frac{\sqrt{2}-\sqrt{2}i}{2}$. We may then use Lemma \ref{rationalmaplemma} to obtain the first rational map approximation $F_0=F_{u_0,v_0}$.

\noindent\textbf{Pullback}: The pullback of $C_0$ by $g$ traverses marked points in the following ordering: 

\begin{center}$\{C_1(\frac{1}{8}),p_1=C_1(\frac{1}{4}),C_1(\frac{3}{8}),C_1(\frac{7}{16}), p_2=C_1(\frac{1}{2}),C_1(\frac{5}{8}),p_3=C_1(\frac{3}{4}),p_4=C_1(\frac{7}{8}),C_1(\frac{15}{16}), p_5=C_1(0)\},$\end{center} 

\noindent where $C_1(\frac{1}{8})=C_1(\frac{5}{8})$ and $C_1(\frac{7}{16})=C_1(\frac{15}{16})$ are respectively the black and red critical points of the mating $g$. (We should note that similar to the last example, the orientation of $C_1$ is such that the curve forks right whenever it approaches the black critical point and left whenever it approaches the red critical point.) Decimal approximations for the marked points traversed by the pullback  curve $F_0^{-1}\circ \sigma_0(C_0)$ by $F_0$ are as follows:

\begin{center} $\{0,.643594i,1.18921i, \infty,-1,0,-.643594i,-1.18921i,\infty,1\}$.\end{center}

These lists of marked points on $C_1$ and the pullback of $\sigma_0(C_0)$ induce a new embedding of $P_g$, which we will denote $\sigma_1$.

\noindent\textbf{Repeat}: For this step we iteratively assign a new rational map and repeat the pullback step. The map $\sigma_1$ embeds the elements of $P_g$ as follows: $$p_1\mapsto .643594i, p_2\mapsto -1, p_3 \mapsto -.643594i, p_4\mapsto -1.18921i, p_5\mapsto 1 $$. We assign the new rational map by noting the new parameters $u_1=\sigma_1(p_1)=.643594i$ and $v_1(\sigma_1(p_4)=-1.18921i$, and applying Lemma \ref{rationalmaplemma} to obtain the approximation $F_1=F_{u_1, v_1}$.

Unlike the previous example, we will have a nontrivial sequence of rational map approximations, since $\sigma_0$ is not a representative for the fixed point of $\Sigma_g$. Continuing to iterate the pullback process from $C_1$ on generates the collection of curves $C_n$ as depicted in Figure \ref{iterates}.

\begin{figure}\label{aiterates}
\includegraphics[width=1.5in]{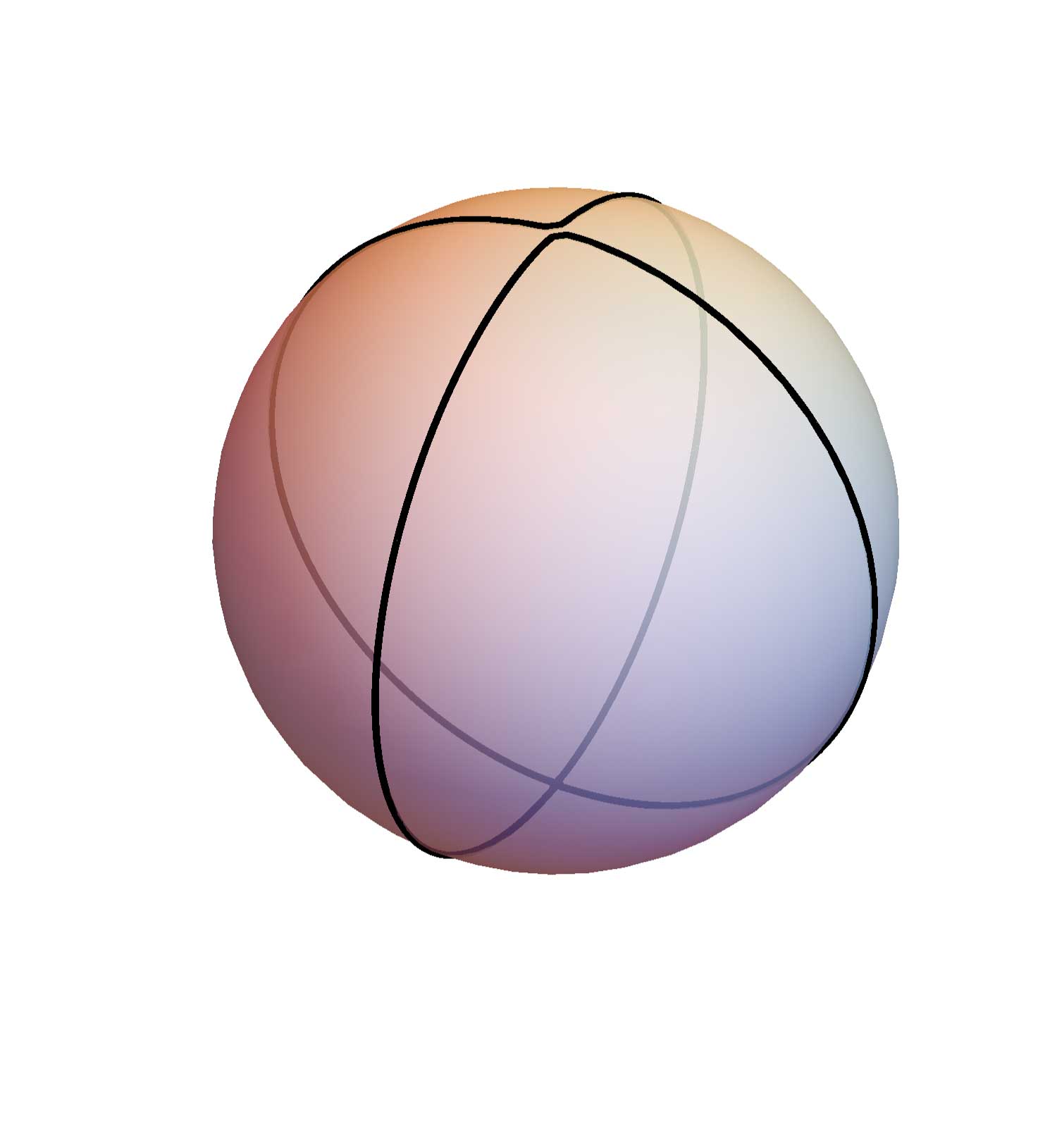}\includegraphics[width=1.5in]{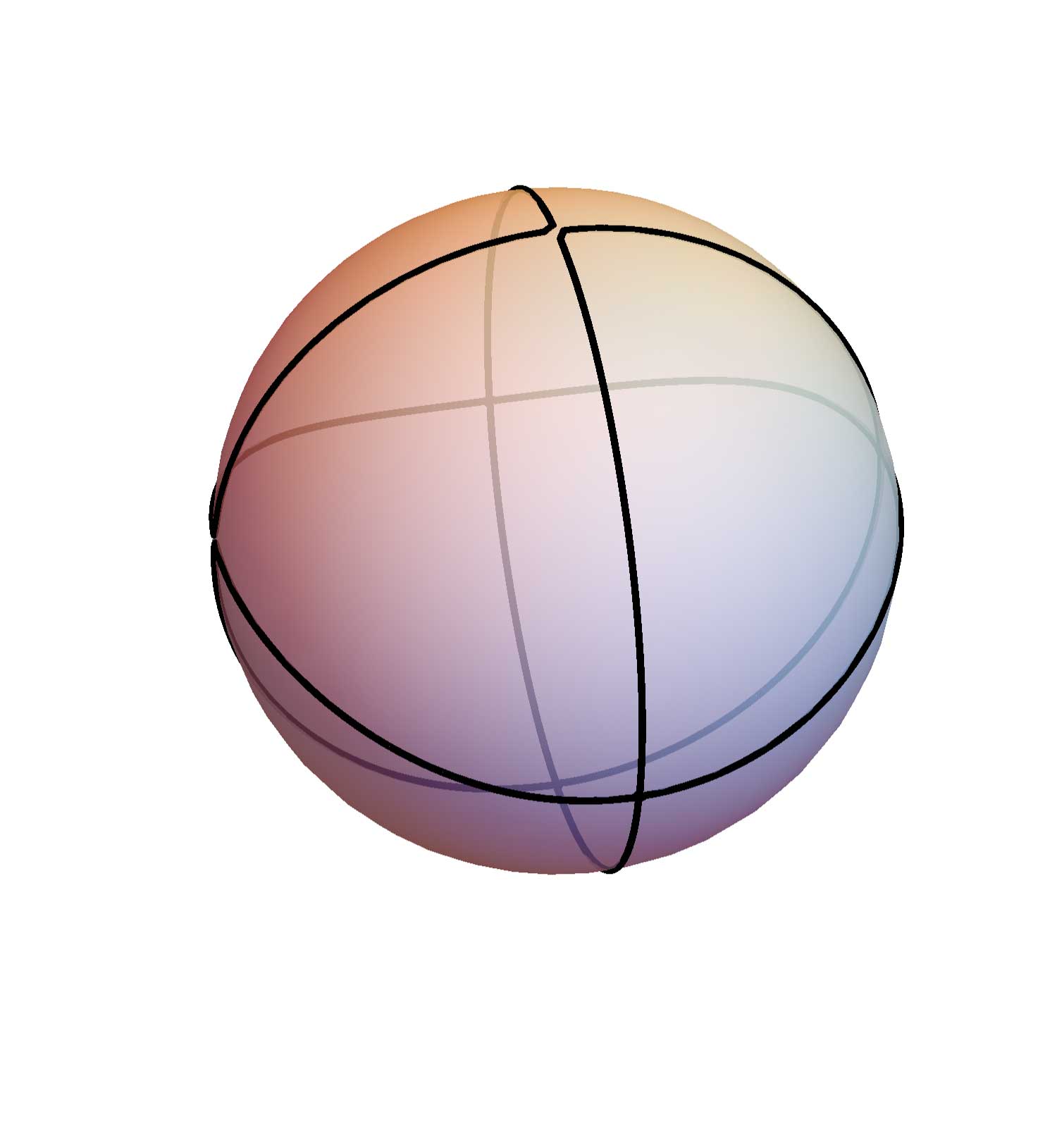}\includegraphics[width=1.5in]{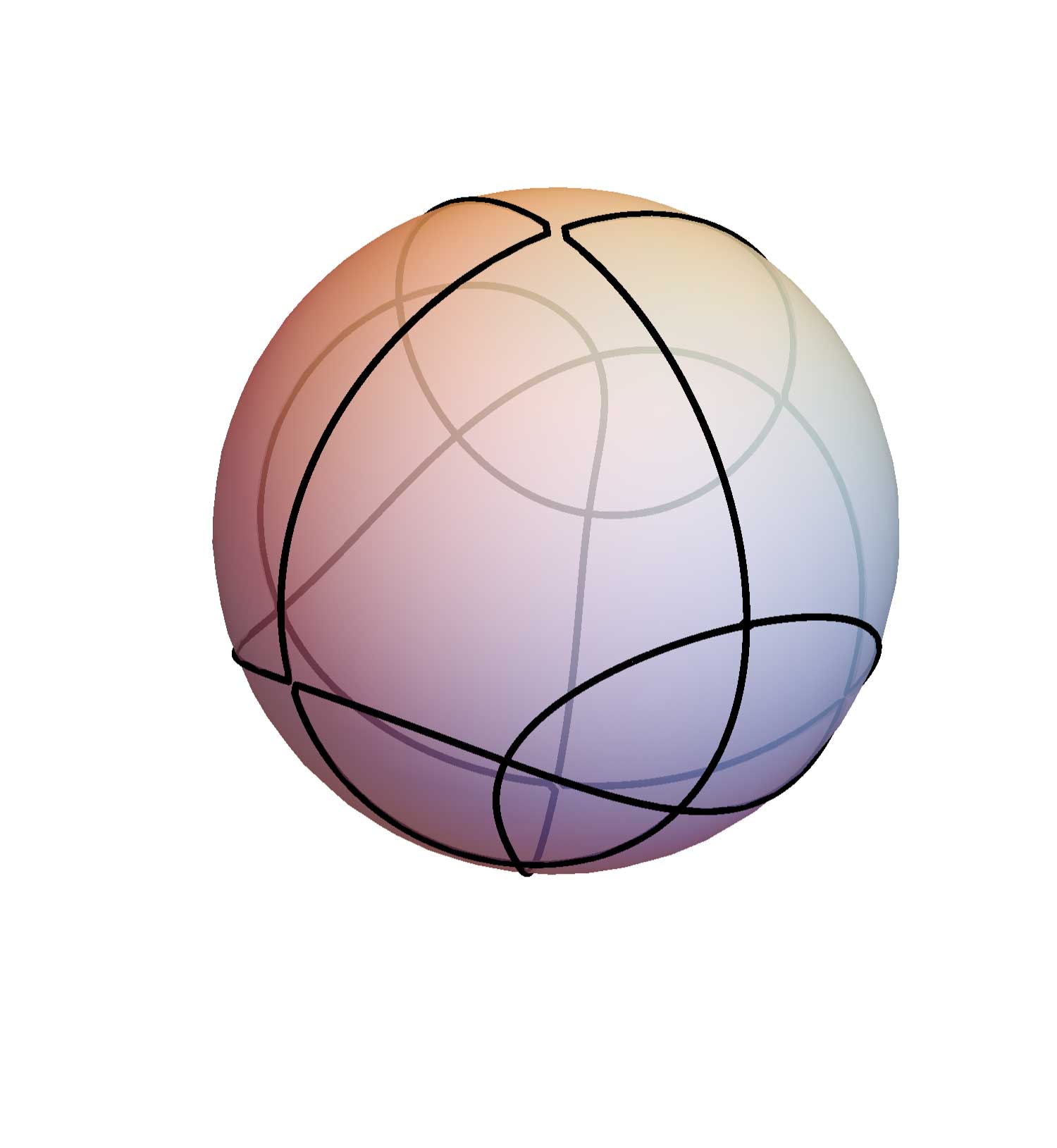}\includegraphics[width=1.5in]{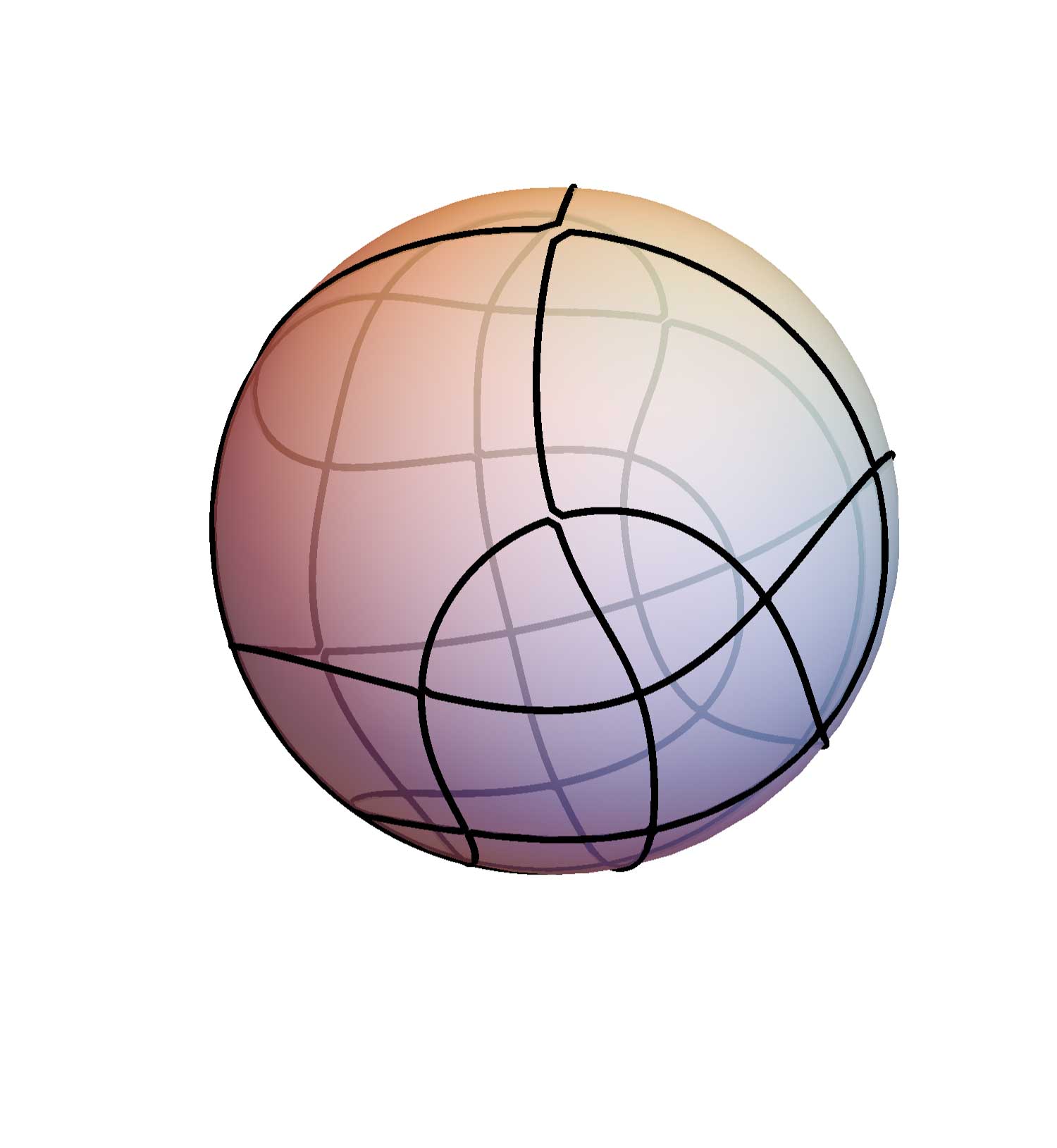}

\includegraphics[width=1.5in]{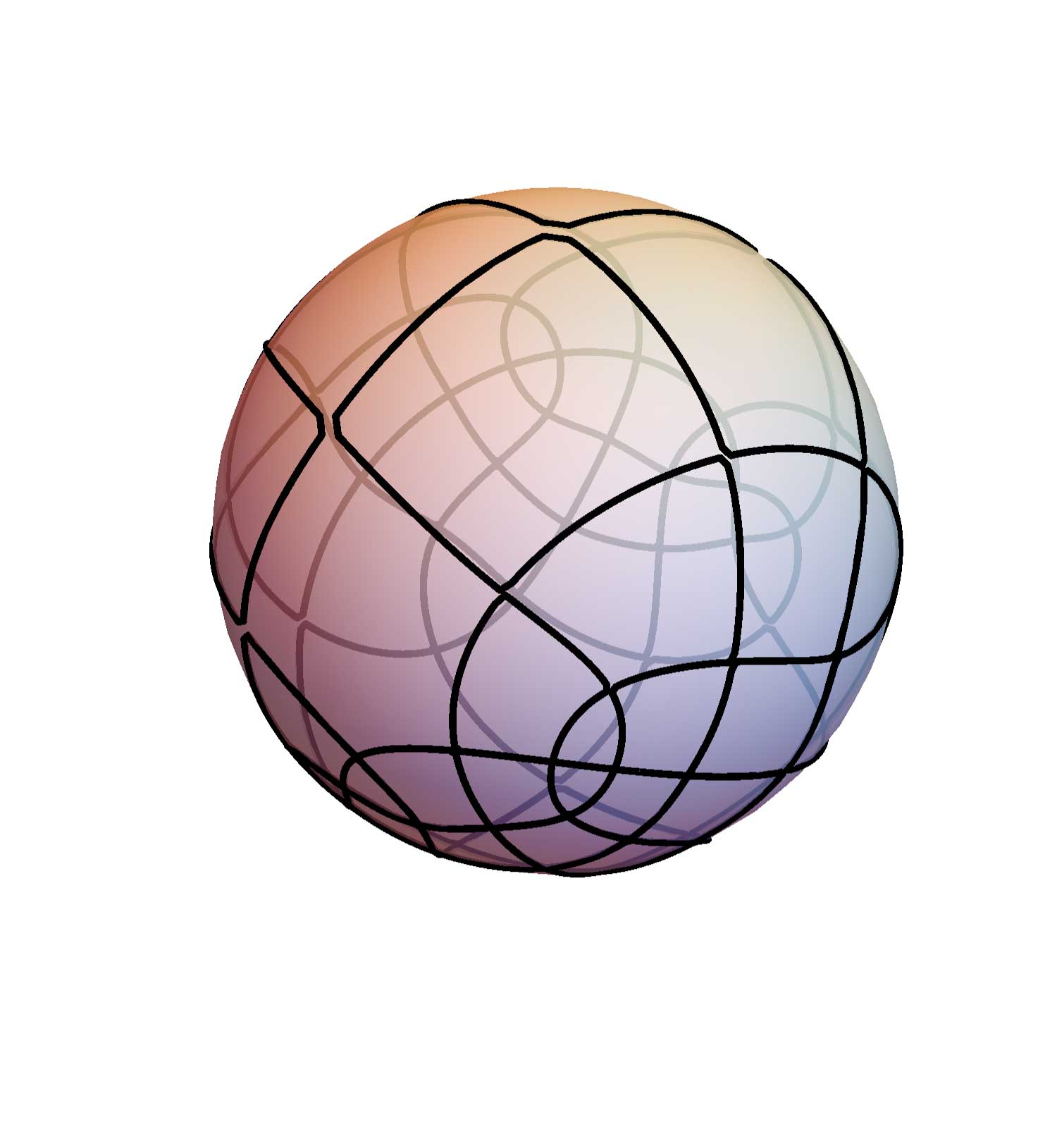}\includegraphics[width=1.5in]{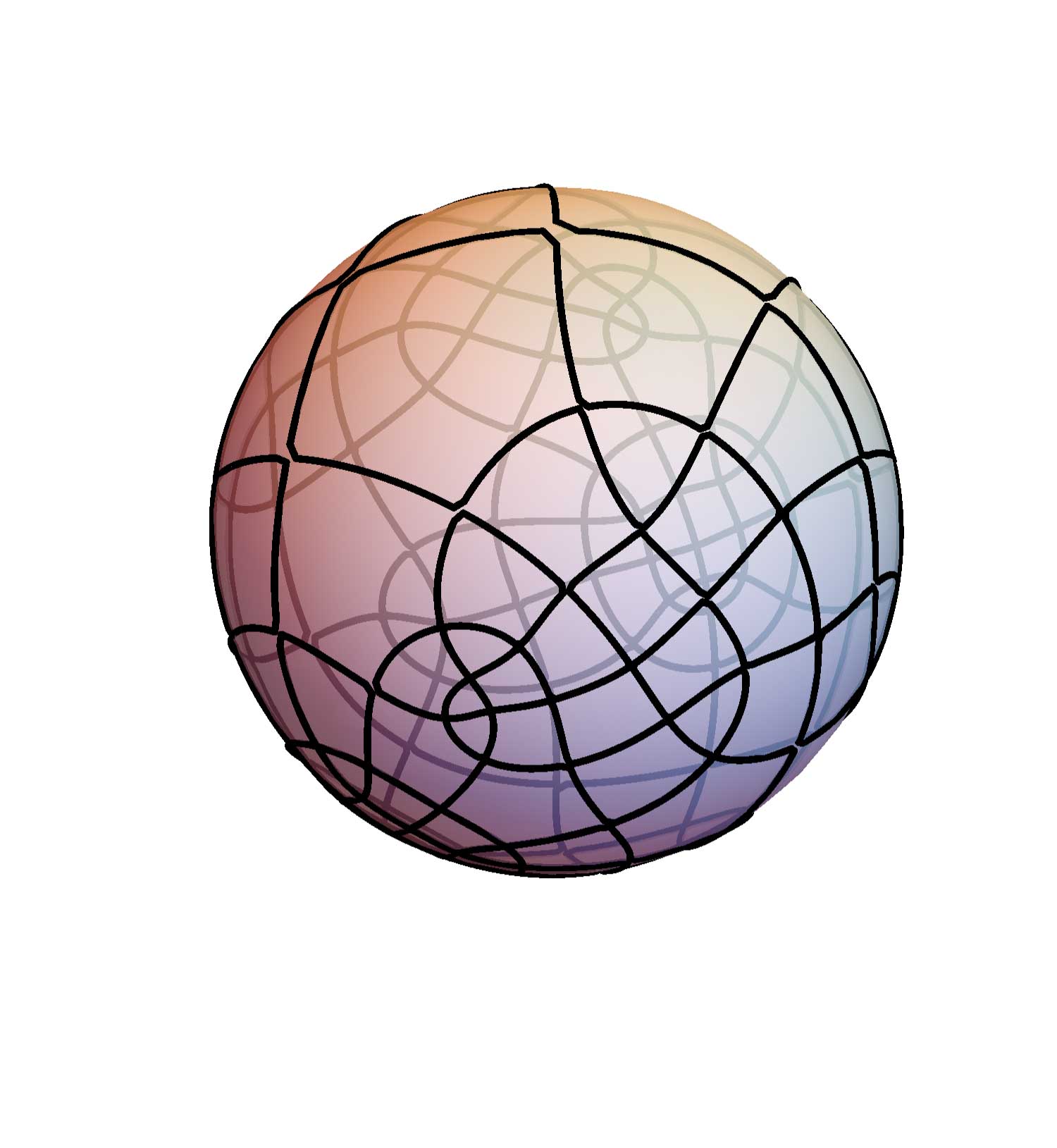}\includegraphics[width=1.5in]{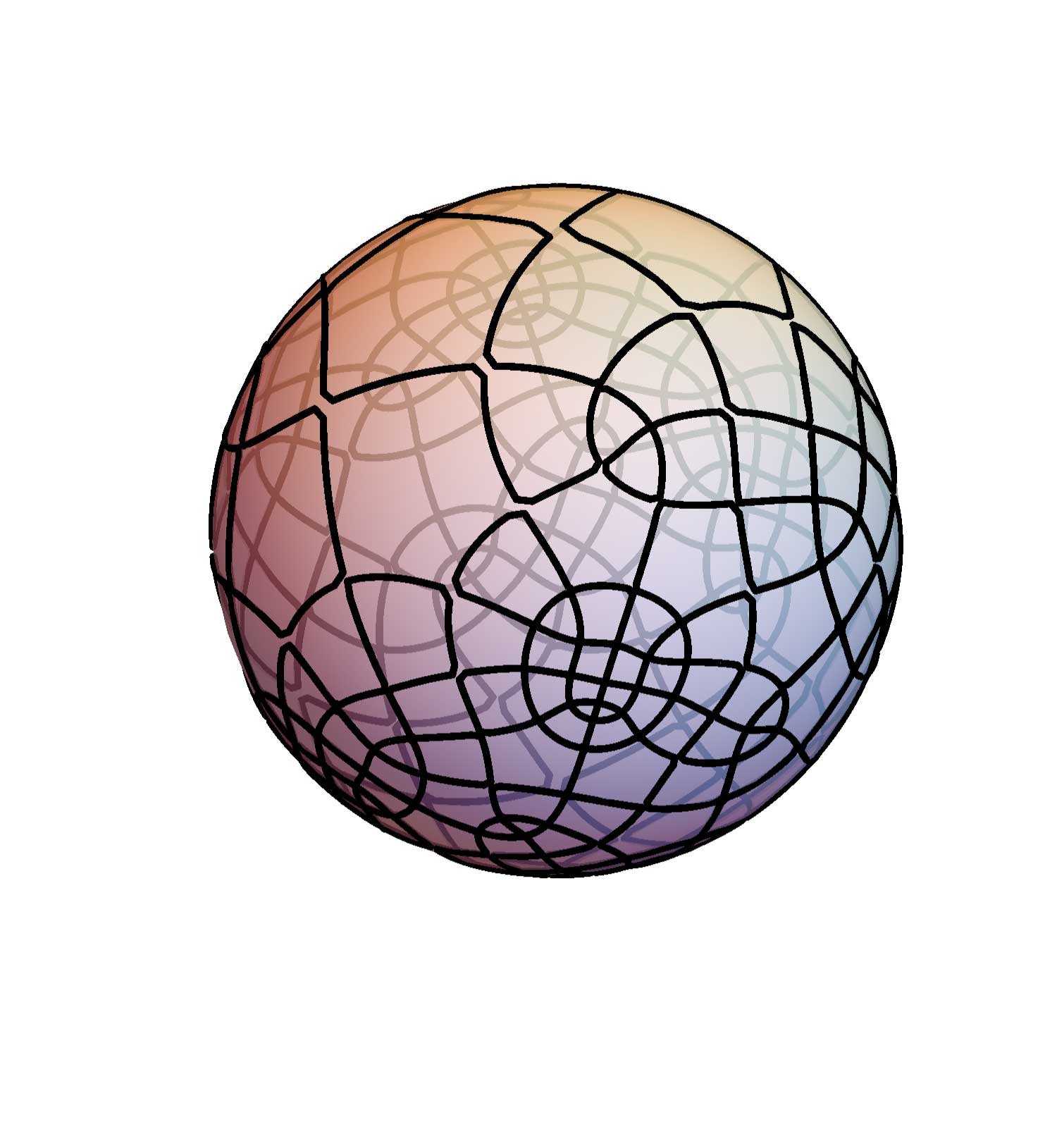}\includegraphics[width=1.5in]{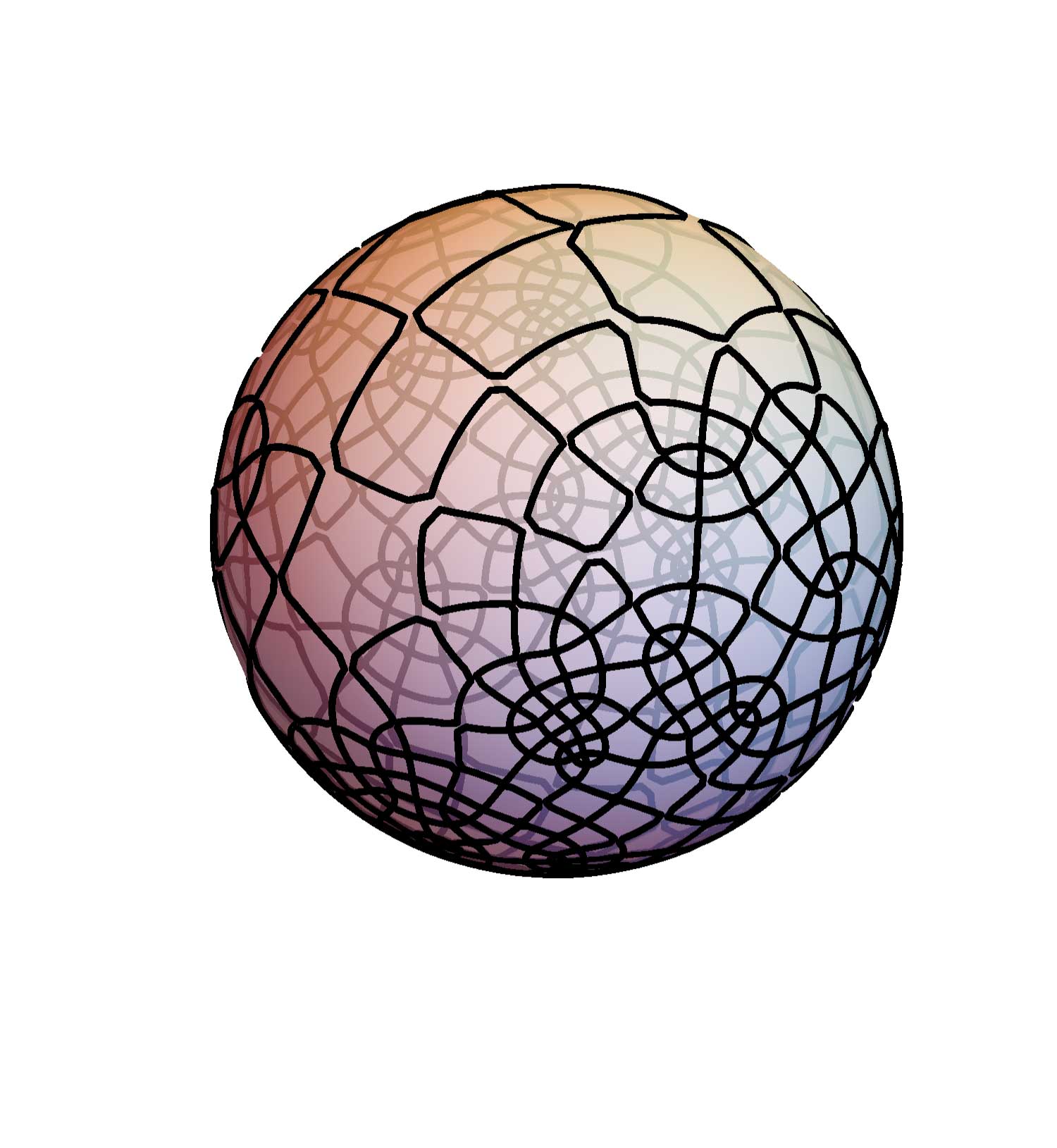}

\includegraphics[width=1.5in]{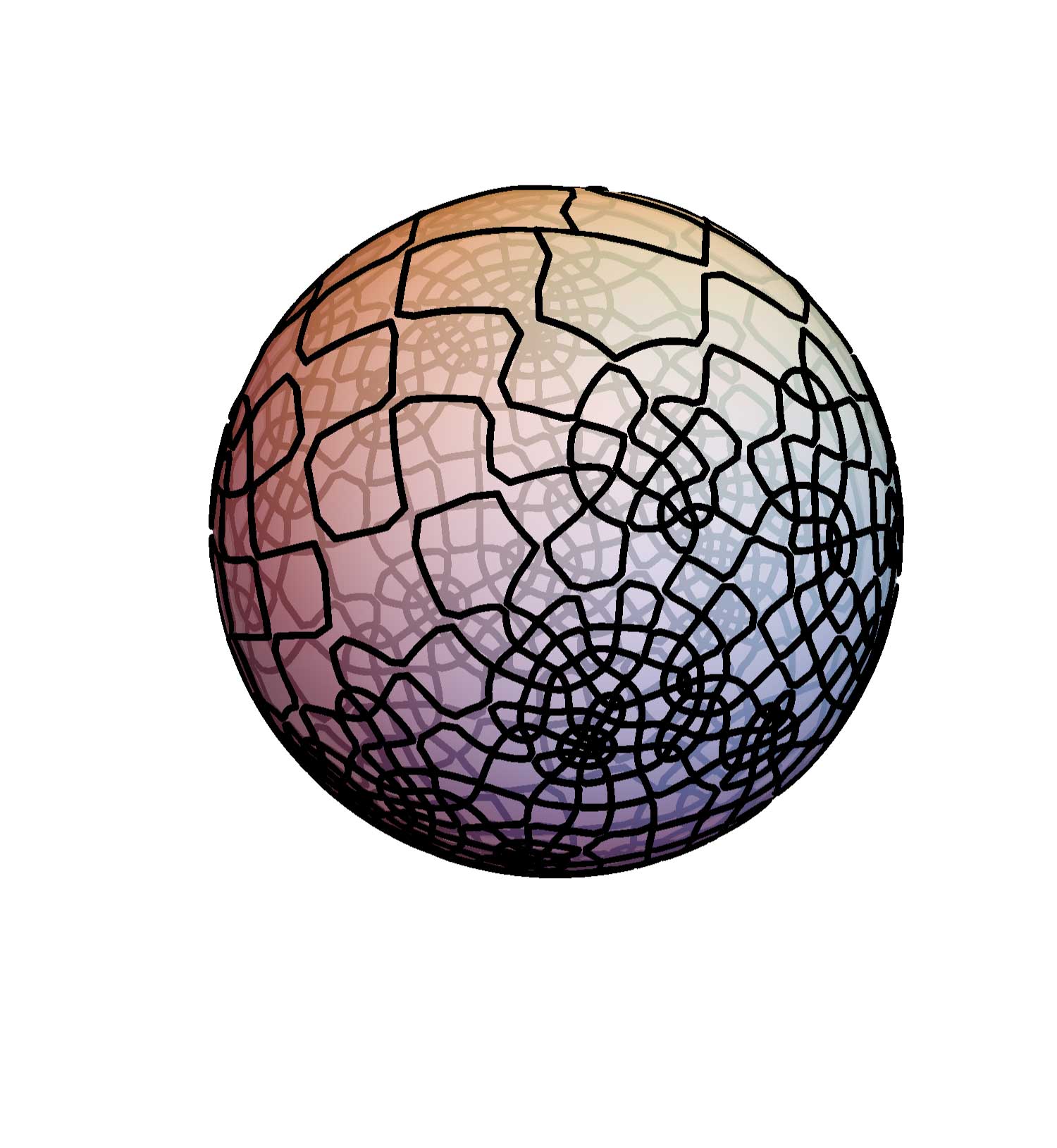}\includegraphics[width=1.5in]{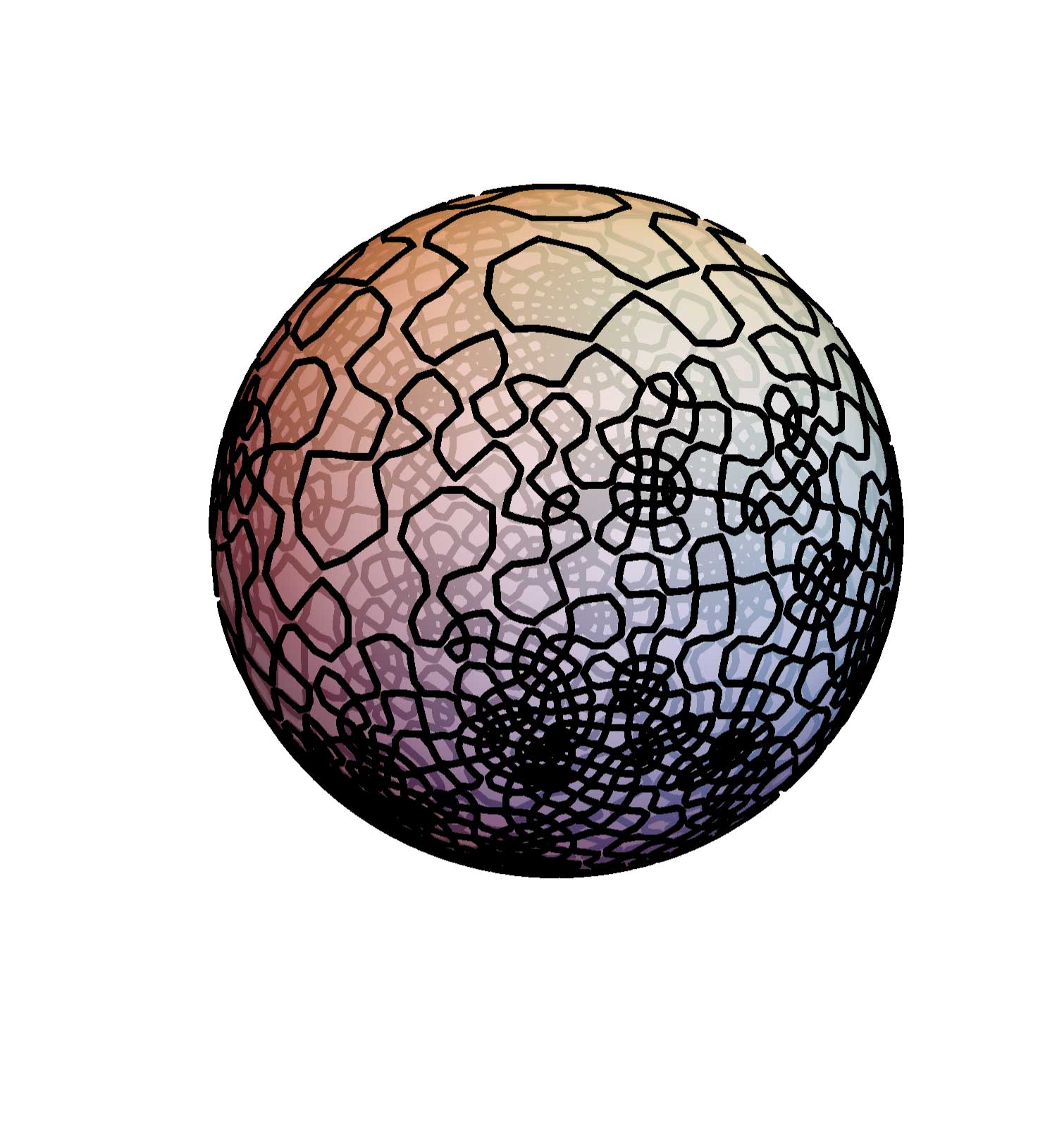}\includegraphics[width=1.5in]{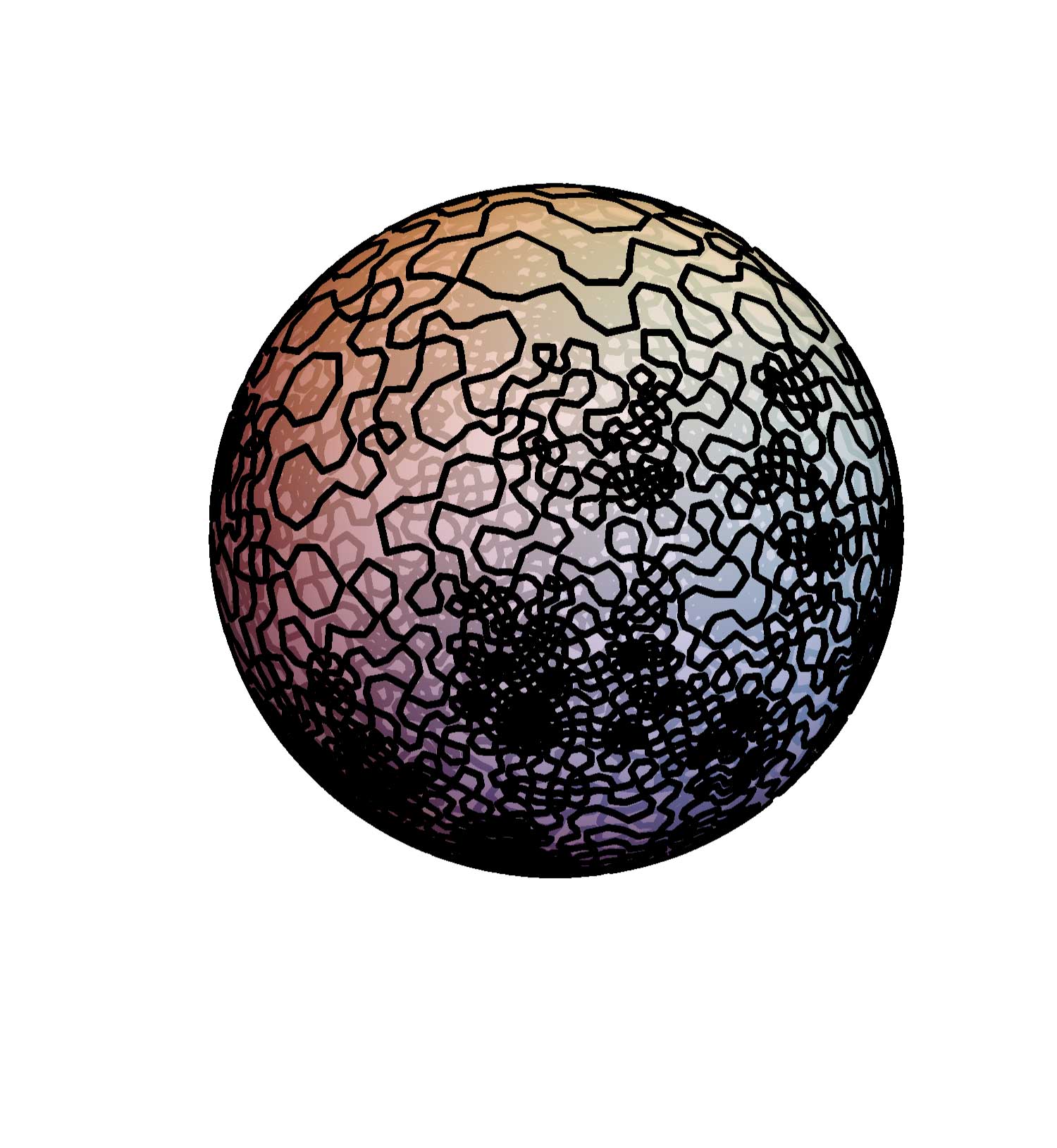}\includegraphics[width=1.5in]{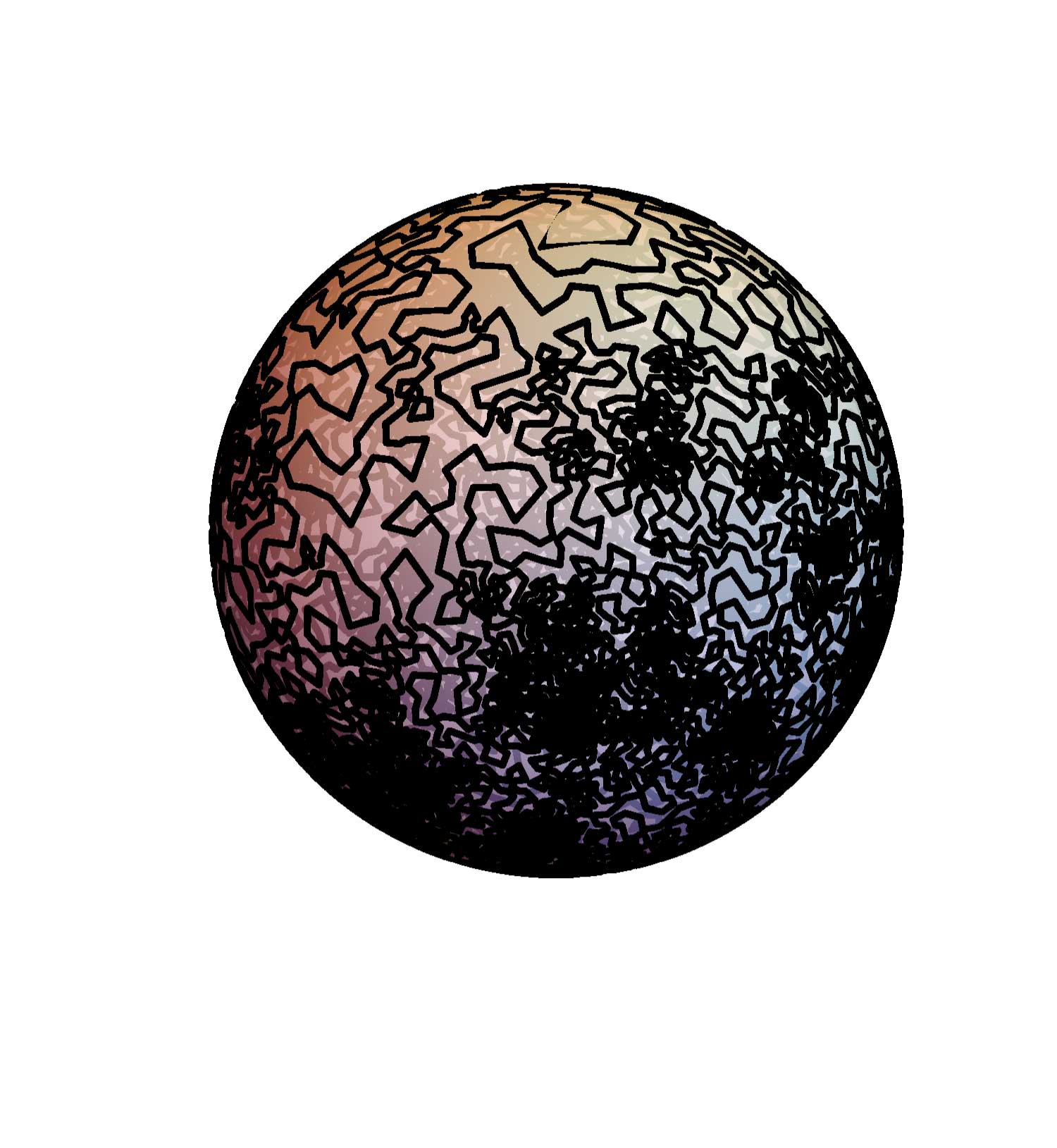}

\caption{Pullbacks of the equator by a sequence of rational maps which approximate the geometric mating of $f_{1/4}$ and $f_{1/8}$.}
\end{figure}

We may note in Figure \ref{aiterates}, similar to the comments at the end of Example \ref{example1}, that `triangular' tile pairs tend to mark the locations of critical values of the map, where in later iterations they visually appear to be a source of prominent dark spots and curve bunching. The settling of darkened spots into relatively stable locations upon later iterations is reflective of the convergence of parameters $u_n$ and $v_n$; thus the convergence of the rational map approximations $F_n$ to our desired geometric mating.

\end{example}

\subsection{Fine tuning and practical concerns}\label{implementation}

In general, we wish to start with a curve containing the postcritical set that separates $0$ and $\infty$. Since we are working with an iterative algorithm, this curve is typically approximated by an ordered list of points in the complex plane--a discretized parameterization of the curve, in a way. The ordering of the points suggests the direction in which we connect the dots to form a Jordan curve, and this direction in which we traverse the curve specifies an orientation. 

Taking a pullback of the curve for our purposes is then somewhat difficult: since the map is 2-1 and the original curve contains the two critical values of the map, the pullback will not be simple. Yet, in order to select a new set of marked points for the next iteration, we are required to have an understanding of the orientation of the pullback. Finding the points in the pullback of the curve is trivial, but how do we ascertain this orientation and maintain a useful discretized parameterization after each pullback?

We must consider the following:

\textbf{Handedness at the critical point:} The pullback is not Jordan, and intersects itself in at least two locations (the critical points). This means that near these points of intersection, we have a choice as to which fork we take to continue along the curve. 

The initial curve that we are using here is actually an approximation to the pseudo equator for the associated rational map. If we think of the pseudo equator as separating our sphere into a black tile and a red tile, the pullback preserves the orientation of the pullback curve with respect to the colors of the pullback tiles. That is, if the curve is positively oriented to black tiles, the pullback curve will be positively oriented to the pullback of the black tiles. Further, the curve initially separates the two Julia sets of polynomials in the mating. We should have a path on the curve to traverse that preserves not only the appropriate orientation, but also 'separates" the Julia sets (as best as possible, since the pullback intersects the critical points).

This suggests that whenever a critical point is being approached as we traverse the pullback, that we are intended to fork a particular direction in order to maintain orientation and return the expected finite subdivision rule. If we stereographically project our complex points, we can check this in a rudimentary way using cross products and the right hand rule.

\textbf{Branches of the square root:}  Using the normalization $\displaystyle F_n(z)= \frac{(u_n-1)v_nz^2-u_n(v_n-1)}{(u_n-1)z^2-(v_n-1)}$, we may think of the pullback of $F_n$ as the composition of a Mobius transformation and the square root function. Since Mobius transformations are orientation preserving on the Riemann sphere, we could apply this map to our discrete curve parameterization and have a direct correspondence with another discrete parameterization of the curve. The square root, however, causes potential problems. We typically think of the pullback by square root as resulting in two pieces: a `positive' version given by the principal branch of the square root function, and a `negative' version given by multiplying the former by negative one. The problem is that depending on our curve, this canonical branch of the square root function may cut our pullback into more than two pieces, jumbling the order in which we wish to traverse the points. A sample problematic curve is depicted in Figure \ref{sqrt}. We thus must pay extra attention to pairs of consecutive points on the curve that lie on opposite sides of the negative real axis. In general, two consecutive points listed on the pullback curve should be near each other, so the appropriate branch of the square root to continue on should be selected accordingly. (In essence, the effect is similar to picking a `smart' branch cut of the square root function which only intersects our curve in one spot, and \emph{then} we can worry about adjoining the positive and negative pieces of our curve for the pullback.)

\begin{figure}\label{sqrt}
\center{\includegraphics[height=3in]{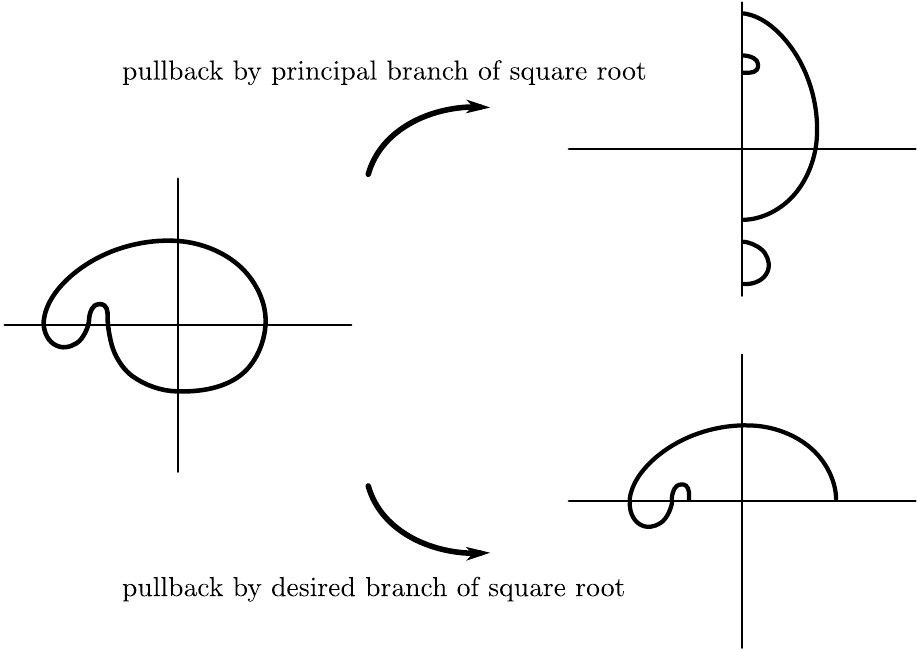}}
\caption{The problem with using the canonical branch of the square root for pullbacks of $C_n$: orientation is important, but harder to keep record of when our pullback curve is cut into several pieces.}
\end{figure}

\textbf{Pruning:} Naively keeping the pullback points obtained in each iteration as a record of $C_n$ will double the amount of data recorded on each iteration. Memory and processing constraints will thus neccessitate simplifying $C_n$ on each iteration before passing through to the next. Since ordering of the postcritical points on $C_n$ is really the most crucial piece of information to record, we can work with pseudo isotopies of $C_n$ instead. One way to do this is to remove points on the curve that do not alter the homotopy type of $C_n$ relative to $P_g$. In \cite{MEDUSA}, a process called ``circlifying" is suggested. In general, the specific method of simplifying $C_n$ while preserving approximate homotopy type is not so important as just implementing \emph{some} way to reduce the data requirements of $C_n$ before moving on to the next step.

\section{Open questions and remarks}
\label{connections}

There are a few clear avenues along which to pursue further investigation, and on which we should make further remarks. A few are highlighted below.

\subsection{Hybrid models and extensions} The focus of this paper has been on quadratic matings in which our initial polynomial pair is not strongly mateable. This is fairly restrictive. To develop a clearer picture of matings, there are two ways that we could attempt a generalization of the algorithm:  working in general quadratic matings, and/or working with matings in higher degrees.

To consider the general quadratic case, it may be helpful to note some key similarities and differences between the Medusa and pseudo-equator models. The pseudo-equator curve is, in essence, a deformation retract of the Medusa. Both models then contain Jordan curve structures that are expected to deform into the Julia set of the mating upon iteration. The curve that does this in the pseudo-equator model contains postcritical points while the analogous curve in the Medusa model does not. In a way, this dichotomy in structure makes sense: we might expect that postcritical points lie on one of these Jordan curve structures if they appear in the Julia set of the mating, and off of them when they appear in the Fatou set. As the Medusa algorithm has difficulty with some cases where the Julia set of the mating is $\mathbb{S}^2$, and the pseudo-equator algorithm is not designed for cases where postcritical points lie in the Fatou set of the mating, this suggests the need for a hybridized Medusa-pseudo-equator model.

To consider the higher degree cases, we could follow this by generalizing the technique given in this paper: obtain a pseudo-equator or Medusa-like structure, find an appropriate finite subdivision rule, apply the subdivision rule to provide mapping instructions for pullbacks, and use an appropriate rational function normalization to apply Thurston's algorithm. This will require some effort in crafting appropriate normalizations and obtaining mapping instructions for pullbacks, since higher degree matings will typically be more complicated maps.

\begin{figure}[hbt]
\center{\includegraphics{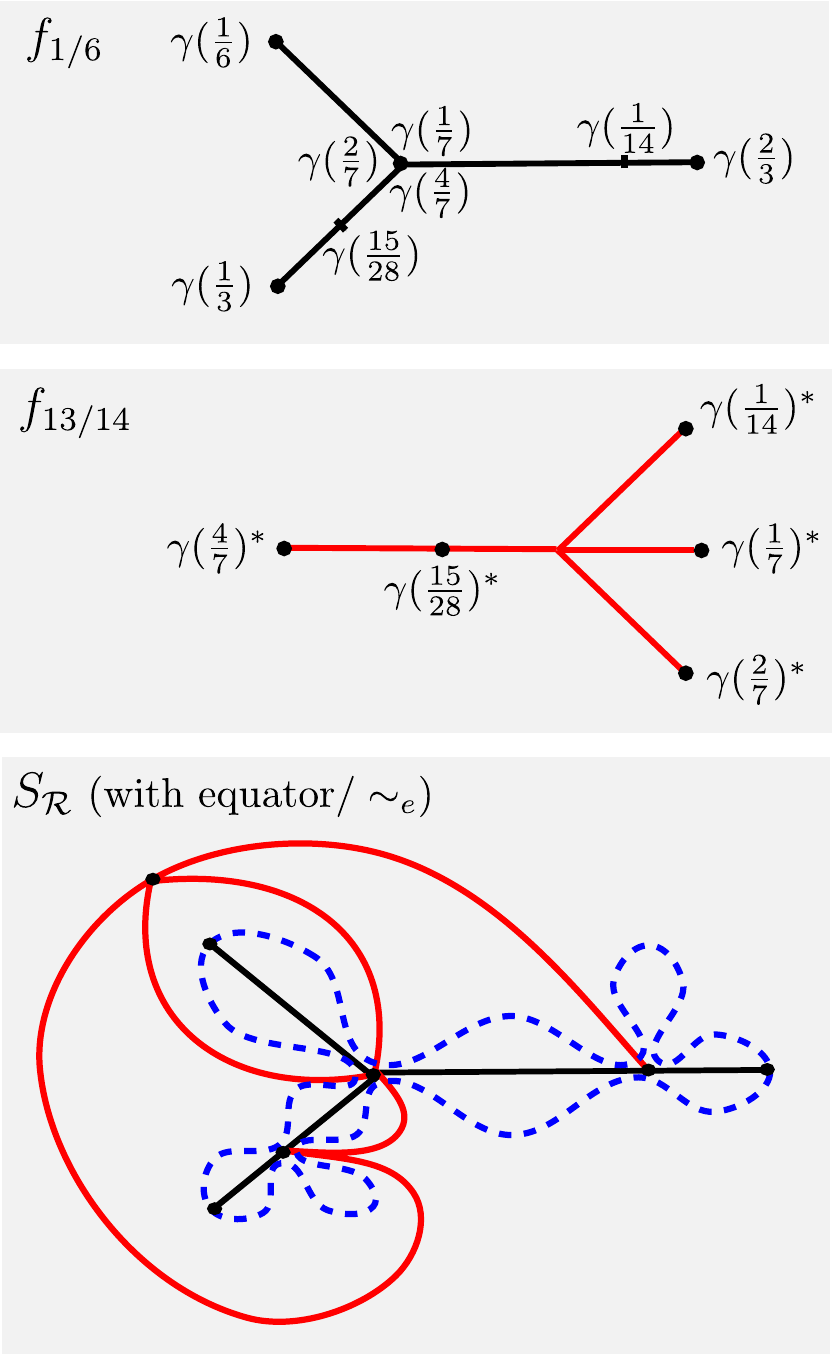}}
\caption{\label{meyerex3}{The ``pseudo-equator" is pinched by $\sim_e$ into a non-Jordan curve.}}
\end{figure}

\subsection{Unmating of rational maps} In \cite{UNMATING}, Meyer comments that the existence of a pseudo equator is sufficient but not necessary to be indicative of a mating. The methods of this paper work in the quadratic cases that a pseudo equator can in some manner be found. Not all non-hyperbolic matings have pseudo-equators, however. A potential reason is that the path $C$ is not always a Jordan curve--any time $\sim_e$ contains equivalence classes that include multiple postcritical or critical points from one of the polynomials in the mating, the equator $\Gamma$ is pinched to form $C$. This falls outside of the scope of the definition for a pseudo equator, which concerns the deformation of a Jordan curve.  For instance, the example given in \cite{UNMATING} for $f_{1/6}\upmodels f_{13/14}$ presents with subdivision complex $S_\mathcal{R}$ and $C$ as shown in Figure \ref{meyerex3}. Notice the pinching of the blue equator curve due to the postcritical identifications on $f_{13/14}$.

We can still use finite subdivision rules to determine the pullbacks of these `pinched' pseudo-equators, very similar to how we have treated pseudo-equators in the rest of this paper. The Caratheodory semiconjugacy is still applicable after this pinching, and will reflect the action of the essential mating of the desired polynomials on $\mathbb{S}^2$. There are still pseudo-isotopies between such pinched curves and their pullbacks, so the general idea of the pseudo-equator algorithm should work even in the case that we examine a non-hyperbolic mating with no pseudo-equator. The computer implementation for such a case would be more difficult, but not impossible.

While pseudo-equators are a sufficient condition for a rational map to be a mating, they are not necessary. This extension of pseudo-equators to pinched pseudo-equators may extend the number of rational maps that could be decomposed as matings though. This suggests the following question: Would a pinched pseudo-equator serve as a sufficient \emph{and} necessary condition to the unmating of a map? If so, and if we tackled the computational details involved in implementation, we could possibly obtain approximations of rational functions for all matings.


\section*{Acknowledgements}

Many of the Julia set graphics throughout the paper were created with the assistance of the dynamics software Mandel 5.8. \cite{MANDEL} The author would also like to extend sincere thanks to Suzanne Boyd for her insights on the workings of the Medusa Algorithm.

\bibliographystyle{plain}
\bibliography{master}

\end{document}